\documentclass[oneside,english]{amsart}
\usepackage[T1]{fontenc}
\usepackage[latin9]{inputenc}
\usepackage{color}
\usepackage{babel}
\usepackage{mathrsfs}
\usepackage{amsbsy}
\usepackage{amstext}
\usepackage{amsthm}
\usepackage{amssymb}
\usepackage{xcolor} 
\usepackage{upgreek}

\usepackage[unicode=true,
 bookmarks=true,bookmarksnumbered=false,bookmarksopen=false,
 breaklinks=false,pdfborder={0 0 1},backref=false,colorlinks=true]
 {hyperref}
\hypersetup{
 linkcolor=red,urlcolor=red,citecolor=red}

\makeatletter

\numberwithin{equation}{section}
\numberwithin{figure}{section}

\theoremstyle{plain}
\newtheorem{thm}{\protect\theoremname}[section]

\theoremstyle{definition}
\newtheorem{defn}[thm]{\protect\definitionname}

\theoremstyle{plain}
\newtheorem{lem}[thm]{\protect\lemmaname}

\theoremstyle{plain}
\newtheorem{prop}[thm]{\protect\propositionname}

\theoremstyle{plain}
\newtheorem*{claim*}{\protect\claimname}

\theoremstyle{remark}
\newtheorem{step}{Step}
\newtheorem{remark}[thm]{\protect\remarkname}

\DeclareMathOperator*{\weaklim}{weak\,lim}

\newcommand{\NN}{\mathbb{N}} 
\newcommand{\RR}{\mathbb{R}} 
\newcommand{\HHH}{\mathcal{H}}

\newcommand{\eps}{\varepsilon}
\newcommand{\indic}{1\!\!1} 

\usepackage[hide]{ed}

\definecolor{pgreen}{rgb}{0.0, 0.5, 0.0}

\title{Scattering for critical radial Neumann waves outside a ball}
\author{Thomas Duyckaerts and David Lafontaine}
\address{Thomas Duyckaerts, LAGA, UMR 7539, Institut Galil\'{e}e, Universit\'{e} Sorbonne Paris Nord, 93430 - Villetaneuse, France}
\email{duyckaer@math.univ-paris13.fr}
\address{David Lafontaine, Department of Mathematical Sciences, University of Bath, Bath, BA2 7AY, United Kingdom}
\email{d.lafontaine@bath.ac.uk}

\thanks{T. Duyckaerts is supported by the Institut Universitaire de France and partially supported by the Labex MME-DII. D. Lafontaine acknowledges support
from EPSRC Grant EP/1025995/1.}

\keywords{Wave equation. Exterior domain. Scattering}

\subjclass[2010]{Primary: 35L70; Secondary: 34B15, 34D05.}

\makeatother

\providecommand{\definitionname}{Definition}
\providecommand{\lemmaname}{Lemma}
\providecommand{\propositionname}{Proposition}
\providecommand{\theoremname}{Theorem}
\providecommand{\remarkname}{Remark}
\providecommand{\claimname}{Claim}

\begin{document}
\maketitle

\begin{abstract}
We show that the solutions of the three-dimensional critical defocusing nonlinear wave equation with Neumann
boundary conditions outside a ball and
 radial initial data scatter. This is to our knowledge the first result of scattering for a nonlinear wave equation with Neumann
boundary conditions. Our proof uses the scheme of concentration-compactness/rigidity introduced by Kenig and
Merle, extending it to our setup, together with 
the so-called channels of energy method  to rule out compact-flow solutions. We also obtain, for the focusing equation,
the same exact scattering/blow-up dichotomy below the energy of the ground-state as in $\mathbb R^3$.
\end{abstract}

\tableofcontents{}

\section{Introduction}
This work concerns the energy-critical wave equation outside an obstacle of $\RR^3$ 
with Neumann boundary condition:
\begin{align}
 \label{NLW}
\partial_t^2u-\Delta u+\iota u^5&=0,\quad \text{in }\Omega\\
\label{NBD}
\partial_nu&=0,\quad \text{in }\partial\Omega\\
\label{ID}
\vec{u}_{\restriction t=0}&=(u_0,u_1)\in \dot{H}^1(\Omega)\times L^2(\Omega),
\end{align}
where $\Omega=\RR^3\setminus K$, $K$ is a compact subset of $\RR^3$ with smooth boundary, $\partial_n u$ is the normal derivative of $u$ on the boundary $\partial \Omega$ of $\Omega$, {$\vec{u}$ denotes $(u, \partial_t u)$,}\ednote{\color{pgreen}à ce stade $\vec{u}$ n'a pas été defini. Comme on l'utilise plusieurs fois dans la suite de l'introduction, j'ai rajouté sa def ici plutôt que de changer dans l'equation} and $\iota\in \{\pm 1\}$. In our main result we will treat the case where $K$ is the unit ball of $\RR^3$ and the initial data $(u_0,u_1)$ is assumed to be radial. 

The equation \eqref{NLW}, \eqref{NBD}, \eqref{ID} is locally well-posed (see \cite{BurqPlanchon09}). The energy
$$\mathscr{E}(\vec{u}(t))=\frac{1}{2}\int_{\Omega} |\nabla u(t,x)|^2dx+\frac{1}{2}\int_{\Omega} |\partial_t u(t,x)|^2dx+\frac{\iota}{6}\int_{\Omega}  u^6(t,x)dx$$
is conserved. When $\iota=1$ (defocusing case), the energy yields a uniform bound of the norm of the solution in $\dot{H}^1\times L^2$ and solutions are expected to be global and to scatter to linear solutions (see definition below). When $\iota=-1$ (focusing case), one can easily construct, using the differential equation $u''=u^5$ and finite-speed of propagation, solutions with initial data in $\dot{H}^1\times L^2$ that blow up in finite time. 

We first consider the defocusing case $\iota=1$. When there is no obstacle ($\Omega=\RR^3$), global existence was obtained for smooth radial data by Struwe \cite{struwe88}, and extended to smooth non-radial data by Grillakis \cite{Grillakis90}. Global existence for data in the energy space was then proved by Shatah and Struwe \cite{ShatahStruwe94}. Bahouri and Shatah \cite{BahouriShatah98} have shown that any solution $u$ to the defocusing equation scatters to a linear solution, i.e. there exists a solution $u_L$ of the free wave equation
\begin{equation}
\label{LW}
\partial_t^2u_L-\Delta u_L=0 
\end{equation} 
on $\RR\times \RR^3$ such that 
$$\lim_{t\to +\infty} \|\vec{u}(t)-\vec{u}_L(t)\|_{\dot{H}\times L^2}=0.$$
The scattering is proved as a consequence of the fact that the $L^6$ norm of the solution goes to $0$, which is obtained by multipliers techniques involving integration by parts on the wave cone $\{|x|\leq |t|\}$.

The equation \eqref{NLW} with Dirichlet boundary condition:
\begin{equation}
\label{DBD}
u_{\restriction \partial \Omega}=0,\quad 
\vec{u}_{\restriction t=0}=(u_0,u_1)\in \dot{H}^1_0(\Omega)\times L^2(\Omega),
\end{equation}
where $\dot{H}^1_0(\Omega)=\{f\in H^1(\Omega),\; f_{\restriction \partial \Omega}=0\}$,
 was studied in several articles. The global well-posedness is proved in \cite{BurqLebeauPlanchon08}. The local well-posedness follows from a local-in-time Strichartz estimate, which is a direct consequence of a spectral projector estimate of Smith and Sogge \cite{SmithSogge07}. The global well-posedness is obtained by the same arguments as in the case without obstacle, observing that the boundary term appearing in the integration by parts can be dealt with a commutator estimate.

 The asymptotic behaviour of equations \eqref{NLW} and \eqref{LW} with Dirichlet boundary conditions \eqref{DBD} is not known in general, and depends on geometrical assumptions on the obstacle. When $K$ is non-trapping, for the linear equation \eqref{LW}, \cite{MorawetzRalston77} proved the exponential decay of the local energy in odd dimensions, polynomial in even dimensions, for compactly supported initial data. A related estimate is the integrability of the local energy, introduced in \cite{MR2001179} 
\begin{equation}
 \label{energy_bound}
\|(\chi u,\chi \partial_t u)\|_{L^2(\RR,H^1\times L^2)}\lesssim_{\chi} \|(u_0,u_1)\|_{\dot{H}^1\times L^2},
\end{equation} 
where $\chi$ is an arbitrary smooth compactly supported function.
In odd space dimensions, the exponential decay of the local energy was first used by \cite{SmithSoggeNonTrapping} to show global-in-time Strichartz estimates. This result was then extended independently to all space dimensions by \cite{MR2001179} and \cite{Metcalfe}. The general argument of Burq \cite{MR2001179} shows that \eqref{energy_bound}, together with the local-in-time Strichartz estimates, imply global Strichartz estimates. When the obstacle $K$ is moreover assumed to be star-shaped, the same computation as in the article of Bahouri and Shatah \cite{BahouriShatah98} yields that any solution scatter to a linear solution. The only difference with the case without obstacle is that boundary terms appear in the integration by parts. The key point is that when $\Omega$ is star-shaped and $u$ satisfies Dirichlet boundary conditions, these boundary terms come with a good sign, so that the proof is still valid in this case. This argument can be extended to illuminated obstacles, that are generalisations of star-shaped obstacles, as done in \cite{Farah1, Farah2} adapting the multiplier so that the boundary term as the right sign, and in \cite{WavesDavid} showing that it decays to zero. However, the case of a general non-trapping obstacle seems at the moment out of hand due to the rigidity of the Morawetz multiplier arguments used for now. 

Much less is known in the case of Neumann boundary conditions. Note that these boundary conditions are more challenging than the Dirichlet boundary conditions, as they do not make sense in the energy space. Also, the strong Huygens principle is lost in this case (see Proposition \ref{prop:linear_group} below).

Local-in-time Strichartz inequalities for the linear wave equation {with Neumann boundary condition} were obtained by Blair, Smith and Sogge \cite{MR2566711}, and global existence for equation \eqref{NLW}-\eqref{NBD} with $\iota=1$ by Burq and Planchon \cite{BurqPlanchon09}. Exponential decay of the local energy in the three-dimensional case was shown by \cite{Morawetz75}. Combined with the local-in-time Strichartz estimates \cite{MR2566711}, this should lead to global in time Strichartz estimates by the arguments of \cite{SmithSoggeNonTrapping}. We give a direct proof of \eqref{energy_bound} (see Proposition \ref{prop:loc_decay}) when the obstacle is the unit ball and the solution is radial, which implies global Strichartz estimates by the main result of
\cite{MR2001179}.

The asymptotic behaviour of the solutions\ednote{Thomas: pour homog\'en\'eiser nos notations, j'ai remplac\'e ``non-linear'' par ``nonlinear'' partout (on utilisait les deux avec un peu plus de ``nonlinear''} of the nonlinear equation \eqref{NLW}-\eqref{ID} was to our knowledge  not {previously} investigated. 
Assuming the global Strichartz estimates for the linear wave equation, the proof of scattering in \cite{BahouriShatah98} does not work anymore since the boundary terms appearing in the integration by parts do not seem to have any specific signs and cannot be controlled.

The main result of this article  is that the scattering to a linear solution holds for the defocusing wave equation with Neumann boundary conditions, when $K$ is the unit ball of $\RR^3$ and $(u_0,u_1)$ is radially symmetric. We thus consider the equation
\begin{align}
 \label{NLWrad}
\partial_t^2u-\Delta u+u^5&=0,\quad \text{in }\RR^3\setminus B(0,1)\\
\label{NBDrad}
\partial_ru&=0,\quad \text{for }r=1\\
\label{IDrad}
\vec{u}_{\restriction t=0}&=(u_0,u_1)\in \mathcal{H}(B^c),
\end{align}
where $B(0,1)$ is the unit ball of $\RR^3$ and $\mathcal{H}(B^c)$ is the space of radial functions in $(\dot{H}^1\times L^2)\left( \RR^3\setminus B(0,1) \right)$, and the corresponding linear wave equation:
\begin{align}
 \label{LWrad}
\partial_t^2u_L-\Delta u_L&=0,\quad \text{in }\RR^3\setminus B(0,1).
\end{align}
with the boundary condition \eqref{NBDrad}.

\begin{thm}
\label{thm:main}
 Let $u$ be a solution of \eqref{NLWrad} with Neumann boundary condition \eqref{NBDrad} and initial data \eqref{IDrad}. Then $u$ is global and there exists a solution $u_L$ of \eqref{LWrad}, \eqref{NBDrad}, with initial data in $\HHH {(B^c)}$, such that 
$$\lim_{t\to +\infty} \|\vec{u}(t)-\vec{u}_L(t)\|_{\HHH {(B^c)}}=0.$$ 
\end{thm}
Our proof uses {and extends} the by now standard compactness/rigidity scheme introduced by Kenig and Merle in \cite{KeMe06}, \cite{KeMe08} to study the focusing energy-critical Schr\"odinger and wave equations on $\RR^N$. 
The compactness step consists in constructing, in a contradiction argument, a global nonzero solution $u_c$ of \eqref{NLWrad}, \eqref{NBDrad} such that 
$$\{\vec{u}_c(t),\; t\in \RR\}$$
has compact closure in $\HHH$. The essential tool of this construction is a profile decomposition (in the spirit of the one introduced by Bahouri and G\'erard \cite{BahouriGerard99} on the whole space), describing the defect of compactness of the Strichartz inequality $\|u_L\|_{L^5(\RR,L^{10})}\lesssim \|(u_0,u_1)\|_{\HHH}$ for solutions of \eqref{LWrad}, \eqref{NBDrad}. We construct this profile decomposition, which is new for the wave equation with Neumann boundary conditions,  in Section \ref{section:profiles}. In this decomposition, the linear wave equation on the whole space appears as a limiting equation for dilating profiles, as shown in Section \ref{section:dilating}. The knowledge of the fact that any solution of the defocusing equation on the whole space scatters is essential to rule out these profiles and obtain the critical solution $u_c$, constructed in Section \ref{section:critical}.

The second step of the proof (the rigidity argument), carried out in Section \ref{section:rigidity} consists in ruling out the existence of the critical solution. Since no monotonicity formula is available due to the Neumann boundary condition, we use the channels of energy method introduced in \cite{DuKeMe11a}, \cite{DuKeMe13} to classify solutions of the focusing energy-critical wave equation on $\RR^3$. Using this method, we prove that $u_c$ must be independent of time, a contradiction with the well-known fact that there is no stationary solution of \eqref{NLWrad} with boundary conditions \eqref{NBDrad} in $\dot{H}^1$. This idea was first used in the context of the supercritical wave equation in \cite{DuKeMe12c}. 

Our method also gives scattering for solutions of the focusing wave equation:
\begin{equation}
 \label{NLWrad_foc}
\partial_t^2u-\Delta u-u^5=0,\quad \text{in }\RR^3\setminus B(0,1),
\end{equation}
with Neumann boundary condition \eqref{NBDrad}
 below a natural energy threshold. Let $W$ be the ground state of the equation on $\RR^3$:
\begin{equation*}
 W=\left( 1+|x|^2/3 \right)^{-1/2}
\end{equation*}
and recall that $W$ is (up to scaling and sign change), the unique radial, stationary solution of $-\Delta W=W^5$ \cite{Pohozaev65, GiNiNi81}. Denote by $\mathscr{E}_{\RR^3}(W,0)$ the energy\ednote{Thomas: dans la suite j'ai not\'e toutes les \'energies non-lin\'eaires avec un $\mathscr{E}$ pour \^etre coh\'erent avec la notation du chapitre $5$. Je n'ai pas indiqu\'e ces changements} of the solution $(W,0)$ on the whole space $\RR^3$:
\begin{equation*}
\mathscr{E}_{\RR^3}(W,0):=\frac{1}{2}\int_{\RR^3}|\nabla W|^2-\frac{1}{6}\int_{\RR^3} W^6. 
\end{equation*} 
Then we have the following:
\begin{thm}
\label{thm_foc1}
 Let $u$ be a solution of \eqref{NLWrad_foc}, \eqref{NBDrad} with initial data \eqref{IDrad}. Assume
$$ \mathscr{E}(u_0,u_1)<\mathscr{E}_{\RR^3}(W,0),\quad \int_{\RR^3\setminus B(0,1)} |\nabla u_0(x)|^2dx<\int_{\RR^3} |\nabla W(x)|^2dx.$$
Then $u$ is global,
$$\forall t\in \RR, \quad \int_{\RR^3\setminus B(0,1)}  |\nabla u(t,x)|^2dx<\int_{\RR^3} |\nabla W(x)|^2dx,$$
and $u$ scatters to a linear solution.
\end{thm}
Finally, we have exactly the same dichotomy as in $\RR^3$ for the solutions below the energy threshold $\mathscr{E}_{\RR^3}(W,0)$, indeed, with the same proof as in \cite{KeMe08}, one obtains:
\begin{thm}
\label{thm_foc2}
 Let $u$ be a solution of \eqref{NLWrad_foc}, \eqref{NBDrad} with initial data \eqref{IDrad}. Assume
$$ \mathscr{E}(u_0,u_1)<\mathscr{E}_{\RR^3}(W,0),\quad \int_{\RR^3\setminus B(0,1)} |\nabla u_0|^2dx>\int_{\RR^3} |\nabla W|^2dx.$$
Then $u$ blows up in finite time.
\end{thm}
Noting that by variational arguments (see Proposition \ref{prop:trapping}), using that $W$ is a maximizer to the critical Sobolev inequilality on $\RR^3$, one cannot have $\mathscr{E}(u_0,u_1)<\mathscr{E}_{\RR^3}(W,0)$ and $\int_{\{|x|>1\}} |\nabla u_0|^2=\int_{\RR^3}|\nabla W|^2$, we see that Theorems \ref{thm_foc1} and  \ref{thm_foc2} describe all solutions of \eqref{NLWrad_foc}, \eqref{NBDrad} such that $\mathscr{E}(u_0,u_1)<\mathscr{E}_{\RR^3}(W,0)$. Let us also mention that Theorems \ref{thm_foc1} and \ref{thm_foc2} cannot be generalized to non-symmetric solutions or other domains than the exterior of a ball, see Subsection \ref{subs:geom}.

We finally give a few more references related to this problem. The study of the linear wave equation outside an obstacle was initiated by Morawetz in \cite{Morawetz61}, and considered in the 60's and 70's by Lax and Phillips,  Morawetz,  Ralston and Strauss, and many other contributors: for an extensive discussion, see for example \cite{MorawetzRalston77} and references therein.

For resolvent estimates in general non-trapping geometries, leading in particular to (\ref{energy_bound}), see \cite{Burq02} and references therein. For a general discussion about local energy decay estimates, one can look at the recent paper \cite{BoucletBurq}.

The focusing nonlinear wave equation with a superquintic nonlinearity outside the unit ball of $\RR^3$  with Dirichlet boundary conditions was considered in \cite{DuyckaertsYang19P}.

The nonlinear Schr\"odinger equation outside a non-trapping obstacle with Dirichlet boundary conditions was first considered in \cite{MR2068304}. The scattering for the three-dimensional defocusing cubic Schr\"odinger equation outside a star-shaped obstacle was shown by Planchon and Vega \cite{PV09}, and by the same authors \cite{PlanchonVega12} for the analogous equation in two space dimensions. The energy-critical case outside a strictly convex obstacle in dimension three was treated in \cite{KillipVisanZhang16}. A scattering result for a nonlinear Schrödinger equation in a model case of weakly trapping geometry can be found in \cite{David2pot}. To our knowledge, there is no work on the nonlinear Schrödinger equation outside an obstacle with Neumann boundary conditions.

\subsection*{Notations} We will use the following notations:
\begin{itemize}
\item If $u$ is a function of time and space, $\vec u$ is understood to be $(u,\partial_t u)$.\ednote{Thomas:j'ai d\'eplac\'e cet item, qui \'etait \`a la fin, ici. Pour homog\'en\'eiser, j'ai mis un verbe \`a chaque phrase, et un point \`a la fin de chaque item.}
\item {Conversely, if $\vec u \in \mathcal H(B^c)$, $u$ is understood to be the first component of $\vec u$.}
\item $B(0,R)$ is the ball centered in $0$ of radius $R$, $B=B(0,1)$, $B^c := \mathbb R ^3 \backslash B(0,1)$ is the domain
we are interested in.
\item $S_{\mathbb{R}^{3}}$ and $S_{N}$ are the linear flow of the wave
equation respectively in $\mathbb{R}^{3}$ and in $B^c$
with Neumann boundary condition. If $(u_0,u_1)$ is the initial data
we will denote { by} $S_N(t)(u_0,u_1)$ or $\big( S_N(u_0,u_1)\big)(t)$ the solution of \eqref{LWrad}, \eqref{NBDrad}, \eqref{IDrad} at time $t$, and $\big( S_N(u_0,u_1)\big)(t,r)$ the solution at time $t$, location $x=|r|$. We use similar notations for $S_{\RR^3}$, and the flows $\mathscr{S}_N$ and $\mathscr{S}_{\RR^3}$ defined below. 
The arrowed versions $\vec{S}_{\mathbb{R}^{3}}$
and $\vec{S}_{N}$ denote the flows together with their
first temporal derivative.
\item $\mathscr{S}_{\mathbb{R}^{3}}$ and $\mathscr{S}_{N}$ are the corresponding
nonlinear flows for the defocusing energy critical wave equation (\ref{NLW}).
\item { We will make the following convention: if $(u_0,u_1)\in \HHH(\RR^3)$, $S_N(t)(u_0,u_1)$ and $\mathscr{S}_N(t)(u_0,u_1)$ will denote the flows applied to the restriction of $(u_0,u_1)$ to $B^c$.}
\item {Throughout the paper, we often deal with solutions of linear and nonlinear equations both in $B^c$ 
with Neumann boundary conditions and in $\mathbb R^3$. In such situations, the letter $u$ has been chosen for the Neumann solutions, whereas $v$ stands for $\mathbb R^3$ solutions.}
\item $L^{p}L^{q}:=L^{p}(\mathbb{R},L^{q}(B^c))$.
\item  $\dot{H}^{1}(B^c)$ is the space of radial functions $f\in L^6(B^c)$ such that $|\nabla f|\in L^2(B^c)$.
\item $\mathcal{H}(B^c)$ is
the space of radial functions in $\dot{H}^1(B^c)\times L^{2}(B^c)$.
\item Finally, $\mathcal{H}(\mathbb R ^3)$ is
the space of radial functions in $\dot{H}^{1}(\RR^3)\times L^{2}(\mathbb{R}^{3})$.
\end{itemize}

\section{Preliminaries}

\subsection{The functional setting}

\begin{defn} \label{def:prolong}
We define the extension operator $\mathcal{P}$ from
$\dot{H}^1(B^c)$ to $\dot{H}_{\text{rad}}^{1}(\mathbb{R}^{3})$
by
$$
\mathcal{P}u(r):=\begin{cases}
u(r) & r\geq1,\\
u(1) & r<1,
\end{cases}
$$
which is well-defined since by the radial Sobolev embedding, for $u \in \dot{H}^1(B^c)$, the function $r\mapsto u(r)$ is continuous on $[1,\infty)$. Similarly, we define the extension operator ${\vec {\mathcal P}}$ from $\mathcal H$ to $\mathcal H(\mathbb R ^3)$ by
$$
\mathcal{\vec {\mathcal P}}(f,g)(r):=\begin{cases}
(f(r),g(r)) & r\geq1,\\
(f(1), 0) & r<1.
\end{cases}
$$

\end{defn}

Let us recall that:
\begin{lem}
\label{lem:u(1)HN}For $u\in\dot{H}^1(B^c)$ we have
\begin{equation}
\int_{1}^{\infty}u(r)^{2}dr\leq4\int_{1}^{\infty}u'(r)^{2}r^{2}dr,\label{eq:hardy}
\end{equation}
in particular, for any compact $K\subset B^c$
\begin{equation}
\Vert u\Vert_{L^{2}(K)}\lesssim\Vert u\Vert_{\dot{H}^{1}}.\label{eq:L^2_comp}
\end{equation}
Moreover
\begin{equation}
|u(1)|\lesssim\Vert u\Vert_{\dot{H}^1(B^c)}.\label{eq:u(1)}
\end{equation}
\end{lem}

\begin{proof}
Integrating by parts, we get
\begin{equation}
\label{eq:IPP}
2\int_{1}^{\infty}ru(r)u'(r)dr=-\int_{1}^{\infty}u(r)^{2}dr-u(1)^{2}.
\end{equation}
Note that the integration by part{s} is justified by approximating $u$
by smooth compactly supported functions. Thus
\[
\int_{1}^{\infty}u(r)^{2}dr\leq-2\int_{1}^{\infty}ru(r)u'(r)dr,
\]
and \eqref{eq:hardy} follows by the Cauchy-Schwarz inequality. The
estimate \eqref{eq:L^2_comp} follows immediately. Bounding the left-hand side of \eqref{eq:IPP} by the Cauchy-Schwarz inequality, and using \eqref{eq:hardy}, we obtain \eqref{eq:u(1)}.
\end{proof}
\begin{remark}
\label{R:global}
 With the same proof, one can generalize \eqref{eq:u(1)} to $|u(r)|\lesssim \|u\|_{\dot{H}^1(B^c)}$. This implies readily that a radial solution of the defocusing critical wave equation with Neumann boundary condition \eqref{NLWrad}, \eqref{NBDrad} is uniformly bounded, thus global (giving a short proof of the result of \cite{BurqPlanchon09} in the radial case). Similarly, any radial solution of the focusing equation \eqref{NLWrad_foc}, \eqref{NBDrad}  that is bounded in $\HHH(B^c)$, is global.\ednote{Thomas: nouvelle remarque}
\end{remark}

\subsection{Linear estimates}

In the present radial case, we can derive an explicit formula
for the linear flow:
\begin{prop}[The linear group] \ednote{\color{pgreen}ici et dans la suite, chang\'e $v=ur$ pour $\zeta = ur$ et $\varphi, \psi$ par $\varphi_{+}, \varphi_{-}$ (pas indiqu\'e par des couleurs)}
\label{prop:linear_group}For any $(u_{0},u_{1})\in\mathcal{H}(B^c)$,
we have, for almost every $r\geq 1$ and $t \in \mathbb R$, and for every $r\geq 1$ and $t \in \mathbb R$ if
we have additionally $(u_{0},u_{1})\in C^1 \times C^0$:
\begin{equation} \label{eq:lin_grp_main}
\big(S_{N}(u_{0},u_{1})\big)(t,r)=\frac{1}{r}\left(\varphi_{+}(r-t)+\varphi_{-}(r+t)\right)
\end{equation}
where, denoting $(\zeta_{0},\zeta_{1}):=(ru_{0},ru_{1})$, for $s\geq1$,
\begin{align}
\varphi_{+}(s) & =\frac{1}{2}\zeta_{0}(s)-\frac{1}{2}\int_{1}^{s}\zeta_{1}(\sigma)d\sigma,\label{eq:phirg1}\\
\varphi_{-}(s) & =\frac{1}{2}\zeta_{0}(s)+\frac{1}{2}\int_{1}^{s}\zeta_{1}(\sigma)d\sigma,\label{psirg1}
\end{align}
and, for $s\in(-\infty,1]$ 
\begin{align}
\varphi_{+}(s) & =\int_{1}^{2-s}e^{s+\sigma-2}(\zeta'_{0}(\sigma)+\zeta_{1}(\sigma))d\sigma-\frac{1}{2}\zeta_{0}(2-s)-\frac{1}{2}\int_{1}^{2-s}\zeta_{1}(\sigma)d\sigma+e^{s-1}\zeta_{0}(1),\label{phirl1}\\
\varphi_{-}(s) & =\int_{1}^{2-s}e^{s+\sigma-2}(\zeta'_{0}(\sigma)-\zeta_{1}(\sigma))d\sigma-\frac{1}{2}\zeta_{0}(2-s)+\frac{1}{2}\int_{1}^{2-s}\zeta_{1}(\sigma)d\sigma+e^{s-1}\zeta_{0}(1).\label{eq:psirl1}
\end{align}
Moreover, for $f \in L^1(\mathbb R, L^2 (\mathbb R ^3,  B^c))$ radial, we have, for $t\geq 0$ and $r-t < 1$
\begin{multline} \label{eq:lin_grp_duhamel}
\int_0^t \Big(S_N (0, f(\tau))\Big)(t-\tau, r) \ d\tau= \frac{1}{r} \int_0^{1-r+t}  \Big( \int_1^{2-r+t-\tau} e^{r-t+\tau+\sigma-2} \sigma f(\tau, \sigma) \ d\sigma \\ + \int_{2-r+t-\tau}^{r+t-\tau} \sigma f(\tau, \sigma) \ d\sigma \Big) d\tau 
+ \frac{1}{2r}\int_{1-r+t}^t \int_{r-t+\tau} ^ {r+t-\tau} \sigma f(\tau, \sigma) \ d\sigma \ d\tau.
\end{multline}
\end{prop}

\begin{proof}
Observe that, arguing by density, it suffices to consider smooth $(u_{0},u_{1})$, for which $\partial_r (S_{N}(u_{0},u_{1}))(1,t) = 0$ for all $t\neq 0$. Let us denote $\zeta(t,r)=rS_{N}(u_{0},u_{1})(t,r).$ Then $\zeta$ is solution
of the one dimensional problem 
\begin{gather}
\partial_{t}^{2}\zeta-\partial_{r}^{2}\zeta=0,\label{eq:wd1}\\
{\partial_{r}\zeta-\zeta}_{\restriction r=1}=0 \quad \forall t \neq 0\label{eq:wd1bound}, \\
\zeta_{\restriction t=0}=(ru_{0},ru_{1}).\label{eq:wd1ini}
\end{gather}
By (\ref{eq:wd1}), $\zeta(r)=\varphi_{+}(r-t)+\varphi_{-}(r+t)$. The boundary condition
(\ref{eq:wd1bound}) gives 
\begin{equation}
\forall t,\ \varphi_{+}'(1-t)+\varphi_{-}'(1+t)=\varphi_{+}(1-t)+\varphi_{-}(1+t),\label{eq:wd1eqpsiphi}
\end{equation}
and the initial condition (\ref{eq:wd1ini}) gives (\ref{eq:phirg1})
and (\ref{psirg1}). Then, integrating (\ref{eq:wd1eqpsiphi}) for
$t\geq0$ gives (\ref{phirl1}), and integrating it for $t\leq0$
gives (\ref{eq:psirl1}).
The identity (\ref{eq:lin_grp_duhamel}) is then a straightforward computation.
\end{proof}

As a first consequence of Proposition \ref{prop:linear_group}, we prove that any radial solution of the linear wave equation on $B^c$ with Neumann boundary conditions is asymptotically close to a solution of the linear wave equation on $\RR^3$.

\begin{prop}
\label{prop:linear_scattering} \ednote{\color{pgreen}Chang\'e $\tilde u$ en $v$ pour la coherence des notations, pas indiqu\'e par des couleurs}
 Let $(u_0,u_1)\in \mathcal{H}(B^c)$. Then 
 \begin{equation}
\label{Hardy0}
\lim_{t\to\infty} \left\|\frac{1}{|x|}S_N(t)(u_0,u_1)\right\|_{L^2(B^c)}=0
 \end{equation}
 and there exists $({v}_0,{v}_1)\in \mathcal{H}(\RR^3)$ such that 
 \begin{equation}
  \label{Free=Neumann}
\lim_{t\to+\infty} \left\|\vec{S}_N(t)(u_0,u_1)-\vec{S}_{\RR^3}(t)({v}_0,{v}_1) \right\|_{\HHH(B^c)}=0.
  \end{equation} 
\end{prop}
\begin{proof}
\setcounter{step}{0}
\begin{step}
 We first prove \eqref{Hardy0}. By a straightforward density argument and the conservation of the energy, we can assume that $(u_0,u_1)$ is smooth and compactly supported. We let $\varphi_+$ and $\varphi_-$ be as in Proposition \ref{prop:linear_group}, and
\begin{equation*}
 \bar{\varphi}_+(s)=\varphi_{+}(s)+\frac{1}{2}\int_1^{+\infty}\zeta_1(\sigma)\,d\sigma,\quad
 \bar{\varphi}_-(s)=\varphi_{-}(s)-\frac{1}{2}\int_{1}^{+\infty} \zeta_1(\sigma)d\sigma.
\end{equation*}
By \eqref{eq:lin_grp_main}, 
\begin{equation} \label{eq:lin_grp_main'}
\big(S_{N}(u_{0},u_{1})\big)(t,r)=\frac{1}{r}\left(\bar{\varphi}_+(r-t)+\bar{\varphi}_-(r+t)\right)
\end{equation}
We claim that there exists $C>0$ (depending on $u$) such that 
\begin{equation}
\label{boundbar}
\left|\bar{\varphi}_+(s)\right|+\left|\bar{\varphi}_-(s)\right|\leq Ce^{s} \indic_{s\leq C}.
 \end{equation} 
 Note that \eqref{eq:lin_grp_main'} and \eqref{boundbar} imply easily \eqref{Hardy0}. Using that $\zeta_0$, $\zeta_0'$ and $\zeta_1$ are bounded and compactly supported, the bound of $\bar{\varphi}_+$ in \eqref{boundbar} follows from the fact that if $s\geq 1$,
\begin{equation*}
\bar{\varphi}_+(s)=\frac{1}{2}\zeta_0(s)+\frac{1}{2}\int_s^{+\infty} \zeta_1(\sigma)\,d\sigma
\end{equation*}
and if $s<1$,
\begin{equation*}
\bar{\varphi}_+(s)=
\int_{1}^{2-s}e^{s+\sigma-2}(\zeta'_{0}(\sigma)+\zeta_{1}(\sigma))d\sigma-\frac{1}{2}\zeta_{0}(2-s)+\frac{1}{2}\int_{2-s}^{+\infty}\zeta_{1}(\sigma)d\sigma+e^{s-1}\zeta_{0}(1).
\end{equation*}
The proof of the bound of $\bar{\varphi}_-$ in \eqref{eq:lin_grp_main'} is very similar and we omit it.
\end{step}
\begin{step}
 We next prove that there exists $({v}_0,{v}_1)\in \mathcal{H}(\RR^3)$ such that \eqref{Free=Neumann} holds. We recall (see e.g. \cite[Theorem 2.1]{DuKeMe19}) that for\ednote{Thomas: quand j'ai fait un changement, j'ai mis en orange la vieille version. Si c'est un rajout, j'ai mis en bleu la nouvelle version} any $G\in L^2(\RR)$, there exists a radial solution $v(t)=S_{\RR^3}(t)({v}_0,{v}_1)$ of the linear wave equation on $\RR^3$ such that 
 \begin{align}
  \label{free_asymptotica}
  \lim_{t\to\infty} \int_0^{+\infty}\left|r\partial_tv(t,r)-G(r-t)\right|^2dr&=0,\\
  \label{free_asymptoticb}
 \lim_{t\to\infty} \int_0^{+\infty}\left|r\partial_rv(t,r)+G(r-t)\right|^2dr&=0.
 \end{align}
Denote by $u(t,r)=\left(S_N(u_0,u_1)\right)(t,r)$. Let $\varphi_{+}(s)$ be as in Proposition \ref{prop:linear_group}. We will prove that $\varphi_{+}'\in L^2(\RR)$ and that 
 \begin{align}
  \label{Neumann_asymptotica}
  \lim_{t\to\infty} \int_1^{+\infty}\left|r\partial_t u(t,r)+\varphi_{+}'(r-t)\right|^2dr&=0,\\
  \label{Neumann_asymptoticb}
 \lim_{t\to\infty} \int_1^{+\infty}\left|r\partial_r u(t,r)-\varphi_{+}'(r-t)\right|^2dr&=0.
 \end{align} 
Letting $(v_0,v_1)$ be such that \eqref{free_asymptotica} and \eqref{free_asymptoticb} are satisfied with $G=-\varphi_{+}'$ we see that \eqref{Neumann_asymptotica} and \eqref{Neumann_asymptoticb} imply the desired conclusion \eqref{Free=Neumann}.

By the definition of $\varphi_{+}$, we have 
\begin{equation*}
  \varphi_{+}'(s)=\begin{cases}
 \frac 12 \zeta_0'(s)-\frac 12 \zeta_1(s)&\text{ if }s\geq 1\\
 -\frac{1}{2}\zeta_0'(2-s)-\frac{1}{2}\zeta_1(2-s)+e^{s-1}u_0(1)&\text{ if }s\leq 1,
              \end{cases}
\end{equation*}
where $(\zeta_0,\zeta_1)=(ru_0,ru_1)$. Since $\zeta_0'$ and $\zeta_1$ are  in $L^2([1,+\infty))$, we obtain that $\varphi_{+}'\in L^2(\RR)$. The same proof yields that $\varphi_{-}'\in L^2(\RR)$.  By Proposition \ref{prop:linear_group}, 
$$\partial_tu(t,r)=\frac{1}{r}\left(-\varphi_{+}'(r-t)+\varphi_{-}'(r+t)\right),$$
and thus
$$\int_{1}^{+\infty}|r\partial_tu(t,r)+\varphi_{+}'(r-t)|^2dr=\int_1^{+\infty}|\varphi_{-}'(t+r)|^2dr \underset{r\to\infty}{\longrightarrow}0.$$
Similarly
$$\partial_ru(t,r)=\frac{1}{r}\left(\varphi_{+}'(r-t)+\varphi_{-}'(r+t)\right)-\frac{1}{r}u(t,r),$$
and thus, using \eqref{Hardy0}, 
\begin{multline*}
\int_{1}^{+\infty}|r\partial_ru(t,r)-\varphi_{+}'(r-t)|^2dr\leq 2\int_1^{+\infty}|\varphi_{-}'(t+r)|^2dr+2\int_{1}^{+\infty}|u(t,r)|^2\,dr\\
\underset{r\to\infty}{\longrightarrow}0. 
\end{multline*}
This concludes the proof.
\end{step}

 \end{proof}

An other consequence of Proposition \ref{prop:linear_group} is the local decay of energy:
\begin{prop}[Local energy decay]
\label{prop:loc_decay}Let $\chi\in C_{c}^{\infty}$. For any $(u_{0},u_{1})\in\mathcal{H}(B^c)$
 
\[
\Vert(\chi u,\chi\partial_{t}u)\Vert_{L^{2}(\mathbb{R},\mathcal{H}(B^c))}\lesssim_{\chi}\Vert(u_{0},u_{1})\Vert_{\mathcal{H}(B^c)}.
\]
where $u=S_{N}(u_{0},u_{1})$.
\end{prop}

\begin{proof}
Let $\zeta(t,r):=ru(t,r)$ and $R>0$ be arbitrary. Note that
\[
\int_1^R r^{2}\left((\partial_{r}u)^{2}+u^{2}+(\partial_{t}u)^{2}\right)dr\lesssim_{R}\int_{1}^{R}\left((\partial_{r}\zeta)^{2}+\zeta^{2}+(\partial_{t}\zeta)^{2}\right)dr,
\]
thus, to obtain the proposition, it is sufficient to show that
\begin{equation}
\int_{-\infty}^{\infty}\int_{1}^{R}\left((\partial_{r}\zeta)^{2}+\zeta^{2}+(\partial_{t}\zeta)^{2}\right)drdt\lesssim_{R}\Vert u_{0}\Vert_{\dot{H}^1(B^c)}^{2}+\Vert u_{1}\Vert_{L^{2}}^{2}.\label{eq:loc_decay_v}
\end{equation}
To this purpose, observe that, by conservation of energy 
\begin{equation}
\int_{-R+1}^{R-1}\int_{1}^{R}\left((\partial_{r}\zeta)^{2}+\zeta^{2}+(\partial_{t}\zeta)^{2}\right)drdt\lesssim_{R}\Vert u_{0}\Vert_{\dot{H}^1(B^c)}^{2}+\Vert u_{1}\Vert_{L^{2}}^{2},\label{eq:loc_dec_comp}
\end{equation}
where we used \eqref{eq:L^2_comp} to bound the $u^{2}$ term. Thus,
it suffices to bound the integrals $\int_{t\geq R-1}$ and $\int_{t\leq-R+1}$.
We will deal with the first one, the proof of the bound for the second
one being similar. Thus, let us suppose that $t\geq R-1$. In particular,
$t\geq r-1$, so $\zeta$ writes, by Proposition \ref{prop:linear_group},
for such $t$'s, for all $r\in [1,R]$:
\begin{multline*}
\zeta(t,r)=\int_{1}^{2-r+t}e^{r-t+\sigma-2}(\zeta'_{0}(\sigma)+\zeta_{1}(\sigma))d\sigma+\frac{1}{2}\int_{2-r+t}^{r+t}\zeta_{1}(\sigma)d\sigma\\
+\frac{1}{2}\zeta_{0}(r+t)-\frac{1}{2}\zeta_{0}(2-r+t)+e^{r-t-1}\zeta_{0}(1).
\end{multline*}
Thus, we have, for $t\geq R-1$ and $1\leq r\leq R$
\begin{multline}
(\partial_{r}\zeta(t,r))^{2}+(\partial_{t}\zeta)^{2}+\zeta^{2}\\ \lesssim_{R}\left(\int_{1}^{2-r+t}e^{\sigma-t}(\zeta'_{0}(\sigma)+\zeta_{1}(\sigma))d\sigma\right)^{2}+\left(\int_{2-r+t}^{r+t}(\zeta'_{0}(\sigma)+\zeta_{1}(\sigma))d\sigma\right)^{2}\label{eq:loc_dec_v}\\
 \ +\zeta'_{0}(2-r+t)^{2}+\zeta'_{0}(r+t)^{2}+\zeta_{1}(r+t)^{2}+\zeta_{1}(2-r+t)^{2} \\
  \ +e^{-2t}\zeta_{0}(1)^{2}.
\end{multline}
By the Cauchy-Schwarz inequality
\begin{multline*}
\left(\int_{1}^{2-r+t}e^{\sigma-t}(\zeta'_{0}(\sigma)+\zeta_{1}(\sigma))d\sigma\right)^{2}  \\
\leq\left(\int_{1}^{2-r+t}e^{\sigma-t}d\sigma\right)\left(\int_{1}^{2-r+t}e^{\sigma-t}(\zeta'_{0}(\sigma)+\zeta_{1}(\sigma))^{2}d\sigma\right)\\
 \lesssim_{R}\int_{1}^{2-r+t}e^{\sigma-t}(\zeta'_{0}(\sigma)^{2}+\zeta_{1}(\sigma)^{2})d\sigma,
\end{multline*}
and therefore,
\begin{multline*}
\int_{R-1}^{\infty}\int_{1}^{R}\left(\int_{1}^{2-r+t}e^{\sigma-t}(\zeta'_{0}(\sigma)+\zeta_{1}(\sigma))d\sigma\right)^{2}d\sigma drdt\\
\lesssim_{R}\int_{R-1}^{\infty}\int_{1}^{R}\int_{1}^{2-r+t}e^{\sigma-t}(\zeta'_{0}(\sigma)^{2}+\zeta_{1}(\sigma)^{2})\ d\sigma drdt\\
\lesssim_{R}\int_{R-1}^{\infty}\int_{1}^{\infty}e^{\sigma-t}(\zeta'_{0}(\sigma)^{2}+\zeta_{1}(\sigma)^{2})\indic_{\sigma\leq2+t}\ d\sigma dt\\
\ =\int_{1}^{\infty}\Big(\int_{R-1}^{\infty}e^{\sigma-t}\indic_{\sigma\leq2+t}dt\Big)(\zeta'_{0}(\sigma)^{2}+\zeta_{1}(\sigma)^{2})d\sigma\\
\ \leq\int_{1}^{\infty}\Big(\int_{\sigma-2}^{\infty}e^{\sigma-t}dt\Big)(\zeta'_{0}(\sigma)^{2}+\zeta_{1}(\sigma)^{2})d\sigma\\
\text{\ \ensuremath{\lesssim\int_{1}^{\infty}}}(\zeta'_{0}(\sigma)^{2}+\zeta_{1}(\sigma)^{2})d\sigma\lesssim\Vert u_{0}\Vert_{\dot{H}^1(B^c)}^{2}+\Vert u_{1}\Vert_{L^{2}}^{2},
\end{multline*}
where we used \eqref{eq:hardy} and the Cauchy-Schwarz inequality to obtain
the last bound. As $\int_{2-r+t}^{r+t}d\sigma\lesssim_{R}1,$ the
term coming from the second term in the first line \eqref{eq:loc_dec_v}
is handled in the same way. Moreover,
\[
\int_{R-1}^{\infty}\int_{1}^{R}\zeta'_{0}(2-r+t)^{2}drdt=\int_{1}^{R}\int_{R-1}^{\infty}\zeta'_{0}(2-r+t)^{2}dtdr\leq R\int_{1}^{\infty}\zeta'_{0}(s)^{2}ds\lesssim_{R}\Vert u_{0}\Vert_{\dot{H}^1(B^c)}^{2},
\]
and all the terms of the second line of \eqref{eq:loc_dec_v} are
dealt with similarly. Finally, the remark that, by Lemma \eqref{eq:u(1)},
\[
\zeta_{0}(1)^{2}=u_{0}(1)^{2}\lesssim\Vert u_{0}\Vert_{\dot{H}^1(B^c)}^{2},
\]
permits to handle the term coming from the third line of \eqref{eq:loc_dec_v}.
We just showed that
\[
\int_{R-1}^{+\infty}\int_{1}^{R}(\partial_{r}\zeta)^{2}+(\partial_{t}\zeta)^{2}+\zeta^{2}\ drdt\lesssim_{R}\Vert u_{0}\Vert_{\dot{H}^1(B^c)}^{2}+\Vert u_{1}\Vert_{L^{2}}^{2}.
\]
Dealing with the part $\int_{-\infty}^{-R+1}$ in the same way and
using \eqref{eq:loc_dec_comp}, the estimate \eqref{eq:loc_decay_v} on $\zeta$, 
and hence the proposition, follow.
\end{proof}

The integrability of the local energy allows us to obtain the following crucial global Strichartz estimates
for the Neumann flow:
\begin{prop}[{Strichartz estimates for the Neumann flow}]
\label{prop:Strichartz}For any couple $(p,q)$ verifying 
\begin{equation}
\frac{1}{p}+\frac{3}{q}=\frac{1}{2},\hspace{1em}\frac{3}{p}+\frac{2}{q}\leq1,\hspace{1em} 2<p\leq \infty,\hspace{1em}q<\infty,\label{eq:strich_adm}
\end{equation}
there exists a constant $C>0$ such that, for all $(u_0, u_1)\in\mathcal{H}(B^c)$
and all $f\in L^{1}(\mathbb{R},L_{\text{rad}}^{2}(r\geq1)),$ if $u$
verifies
\begin{align*}
\partial_{t}^{2}u-\Delta_{N}u & = f,\text{ in }B^c,\\
\partial_{n}u & =0\text{ on }\partial B(0,1),\\
\vec u _{\restriction t=0} & =(u_0, u_1),
\end{align*}
then, for all $T>0$
\[
\Vert u\Vert_{L^{p}([-T,T],L^{q}(r\geq1))}\leq C\Big(\Vert(u_0, u_1)\Vert_{\mathcal{H}(B^c)}+\Vert f\Vert_{L^{1}([-T,T],L_{\text{}}^{2}(r\geq1))}\Big).
\]
\end{prop}

\begin{proof}
The main result of \cite{MR2001179} shows that the local energy decay
of Proposition \ref{prop:loc_decay} combined with local in time Strichartz
estimates implies global in time ones. Such local estimates where
shown by \cite{MR2566711} for the above range of couples $(p,q)$,
hence the proposition follows.
\end{proof}

Let us also recall the Strichartz estimates in $\mathbb R ^3$:

\begin{prop}[{Strichartz estimates in $\mathbb R ^3$, \cite{GV87, GV95}, \cite{LS}, \cite{KT}}]
\label{prop:Strichartz_free}For any couple $(p,q)$ verifying 
\begin{equation}
\frac{1}{p}+\frac{3}{q}=\frac{1}{2},\hspace{1em}\frac{1}{p}+\frac{1}{q}\leq\frac{1}{2},\hspace{1em}2<p\leq \infty,\hspace{1em}q<\infty,\label{eq:strich_adm_free}
\end{equation}
there exists a constant $C>0$ such that, for all $(u_0, u_1)\in\mathcal H(\mathbb R ^3)$
and all $f\in L^{1}(\mathbb{R},L^{2}(\mathbb R ^3)),$ if $v$
verifies
\begin{align*}
\partial_{t}^{2}v-\Delta v & = f,\\
\vec v _{\restriction t=0} & =(v_0, v_1),
\end{align*}
then, for all $T>0$
\[
\Vert v\Vert_{L^{p}([-T,T],L^{q}(\mathbb R ^3))}\leq C\Big(\Vert(v_0, v_1)\Vert_{\mathcal{H}(\mathbb R ^3)}+\Vert f\Vert_{L^{1}([-T,T],L_{\text{}}^{2}(\mathbb R ^3))}\Big).
\]
\end{prop}

\begin{remark}
Observe the loss in the range of admissibles couples (\ref{eq:strich_adm}) in Proposition \ref{prop:Strichartz} compared to the free case (\ref{eq:strich_adm_free}). This is because we used the local-in-time Strichartz estimates of \cite{MR2566711}, valid in a general geometrical setup. It is likely that the above Strichartz estimates, outside a ball, could be extended to the full range of couples (\ref{eq:strich_adm_free}), using for the local-in-time estimates a construction similar to the one done by \cite{SS95} for Dirichlet boundary conditions. However, the range of exponents (\ref{eq:strich_adm}) is sufficient for our analysis and we don't pursue this question here.

\end{remark}

As a last consequence of the explicit formula for the linear group given by Proposition \ref{prop:linear_group}, we have
\begin{lem} \label{lem:interpr_neumann}
Let  $(u_0, u_1) \in (C^1 \times C^0) \cap \mathcal H(B^c)$. Then
\mbox{} \begin{enumerate}
\item we have
$$
S_N(\cdot) (u_0, u_1) \in C^0(\mathbb R \times  B^c) \cap C^1(\{ t \pm r \neq 1 \}),
$$
with
$$
\partial_r \big( S_N (u_0, u_1) \big) (1,t) = 0 \hspace{0.3cm} \forall t \neq 0,
$$
\item if in addition $f\in L^1(\mathbb R, L^2( B^c))$ is radial and continuous and $u$ is defined by
$$
u(t) := S_N(t)(u_0, u_1) + \int_0^t S_N(t-\tau) (0, f(\tau)) \ d\tau,
$$
then $u \in C^0(\mathbb R \times  B^c) \cap C^1(\{ t \pm r \neq 1 \})$ and
$$
\partial_r u (1,t) = 0 \hspace{0.3cm} \forall t \neq 0.
$$
\end{enumerate}
\end{lem}

\begin{proof}
The explicit formulas (\ref{eq:lin_grp_main}), (\ref{eq:phirg1}),  (\ref{psirg1}),  (\ref{phirl1}),  (\ref{eq:psirl1}) give the first part of the lemma, and  (\ref{eq:lin_grp_duhamel}) then gives the second part for $t>0$. The case $t<0$ is given by a similar computation.
\end{proof}

\subsection{Perturbative theory}
\begin{defn}
 \label{def:scattering}
We say that a solution $u$ of the nonlinear wave equation \eqref{NLWrad}, with Neumann boundary conditions  \eqref{NBDrad} \emph{scatters in the future}\ednote{Thomas: ce n'est pas grave, mais l'instruction $\backslash emph $ provoque un soulignement, ce qui n'est pas le cas d'habitude avec amsart. Sais-tu pourquoi? {\color{pgreen} C'était le package ulem, qui permet de barrer, souligner, etc... que j'avais ajouté à un moment pour mettre en avant les modifications, qui modifiait ce comportement... je l'ai enlevé!}} when there exists a solution $u_L$ of the linear wave equation \eqref{LWrad} with Neumann boundary conditions \eqref{NBDrad} such that 
$$\lim_{t\to+\infty}\left\|\vec{u}(t)-\vec{u}_L(t)\right\|_{\HHH (B^c)}=0.$$
We define similarly \emph{scattering in the past}. We say that the solution \emph{scatters} when it scatters both in the future and in the past.
\end{defn}
In a classical way, we have
\begin{prop}
\label{prop:perturb_scat}Let $(u_0,u_1)\in\mathcal{H}(B^c)$ and $u(t)=\mathscr{S}_N(t)(u_0,u_1)$.
\begin{equation}
u\in L^{5}\left([0,+\infty),L^{10}\right)\implies u \text{ scatters in the future}.\label{eq:pert_scat_carac}
\end{equation}
A similar property holds in the past.
Moreover, there exists $\epsilon_{0}>0$ such that, for any $(u_0,u_1)\in\mathcal{H}(B^c)$,
\begin{equation}
\Vert(u_0,u_1)\Vert_{\mathcal{H}(B^c)}\leq\epsilon_{0}\implies\mathscr{S}_{N}(\cdot)(u_0,u_1)\in L^{5}L^{10},\label{eq:pert_scat_small}
\end{equation}
and $\mathscr{S}_{N}(\cdot)(u_0,u_1)$ scatters.
 And, for any $(u_0,u_1)\in\mathcal{H}(B^c)$,
 there exists a solution $U^{\pm} \in L^5(\mathbb R_{\pm}, L^{10})$ of (\ref{NLWrad})-(\ref{NBDrad}) such that
 \begin{equation}
 \Vert \vec{U}^{\pm}(t)-\vec{S}_{N}(t)(u_0,u_1)\Vert_{\mathcal{H}(B^c)}\longrightarrow0,\text{ as }t\longrightarrow\pm\infty.\label{eq:pert_scat_Uj}
 \end{equation}
\end{prop}

\begin{proof}[{Sketch of proof}]
Observe that $(5,10)$ is Strichartz-admissible in the sense of Proposition
\ref{prop:Strichartz}. The properties \eqref{eq:pert_scat_carac} and \eqref{eq:pert_scat_small}
are then classical consequences of the global in time Strichartz estimates.
Finally, \eqref{eq:pert_scat_Uj} can be proved by a fixed point argument using the Strichartz estimates.
\end{proof}
In addition,
\begin{prop}[Perturbation] \ednote{\color{pgreen}chang\'e $u$ et $v$ pour $u, \tilde u$ et amélioré la presentation}
\label{lem:perturb}For any $M>0$, there exists $\epsilon(M)>0$
such that, for any $0<\epsilon\leq\epsilon(M)$, and all $(u_0, u_1), \,(\tilde u_0, \tilde u_1)\in\mathcal{H}(B^c)$,
$e\in L^{1}L^{2}$ and $u\in L^{5}L^{10}$ verifying 
\[
\Vert u\Vert_{L^{5}L^{10}}\leq M,\hspace{1em}\Vert S_N(\cdot)\big((u_0, u_1)-(\tilde u_0, \tilde u_1)\big)\Vert_{L^5L^{10}}\leq\epsilon,\hspace{1em}\Vert e\Vert_{L^{1}L^{2}}\leq\epsilon,
\]
if $u,\tilde u$ are solutions of 
\begin{equation*}
\begin{tabular}{cc}
$
\left\{
\begin{aligned}
\partial_{t}^{2}u-\Delta_{N}u&=-u^{5}\text{ in }B(0,1)^{c}, \\
\vec u_{\restriction t=0}&=(u_0,u_1),\\
\partial_{n}u&=0\text{ on }\partial B(0,1),
\end{aligned} \right.
$
&
$
\left\{
\begin{aligned}
\partial_{t}^{2}\tilde u-\Delta_{N}\tilde u&=-\tilde u^{5} + e\text{ in }B(0,1)^{c}, \\
\vec {\tilde u}_{\restriction t=0}&=(\tilde u_0,\tilde u_1),\\
\partial_{n}\tilde u&=0\text{ on }\partial B(0,1),
\end{aligned}\right.
$
\end{tabular}
\end{equation*}
then $\tilde u\in L^{5}L^{10}$ and we have 
\[
\Vert u-\tilde u\Vert_{L^{5}L^{10}}\lesssim\epsilon.
\]
In addition, the same statement holds for the corresponding equations in $\mathbb{R}^{3}$.
\end{prop}

\begin{proof}
The proof is classical and similar to Proposition 4.7 of \cite{MR2838120}
, we give it for completeness. Let us denote $w=u-\tilde u$. Then $w$ is
solution of
\begin{gather*}
\partial_{t}^{2}w-\Delta_{N}w=-u^{5}+\tilde u^{5}-e,\hspace{1em}\vec w_{\restriction t=0}=(u_0, u_1)-(\tilde u_0, \tilde u_1).
\end{gather*}
Let $T>0$. By the Strichartz inequality for the Neumann flow (Proposition
\ref{prop:Strichartz}) applied to $w$, we get, with an implicit constant
independent of $T$
\begin{align*}
\Vert u-\tilde u\Vert_{L^{5}(-T,T)L^{10}} & \lesssim\Vert \tilde u^{5}-u^{5}\Vert_{L^{1}(-T,T)L^{2}}+\Vert e\Vert_{L^{1}L^{2}}\\ & \hspace{1em}+\Vert S_N(\cdot)((u_0, u_1)-(\tilde u_0, \tilde u_1))\Vert_{L^5L^{10}}\\
 & \lesssim\Vert|u-\tilde u|(|u|^{4}+|u-\tilde u|^{4})\Vert_{L^{1}((-T,T)L^{2})}+\Vert e\Vert_{L^{1}L^{2}}\\
 & \hspace{1em}+\Vert S_N(\cdot)((u_0, u_1)-(\tilde u_0, \tilde u_1))\Vert_{L^5L^{10}}\\
 & \leq C\Big(\int_{-T}^{T}\Vert u-\tilde u\Vert_{L^{10}}\Vert u\Vert_{L^{10}}^{4}+\Vert u-\tilde u\Vert_{L^{5}((-T,T),L^{10})}^{5}\\
 & \hspace{1em}+\Vert e\Vert_{L^{1}L^{2}}+\Vert S_N(\cdot)((u_0, u_1)-(\tilde u_0, \tilde u_1))\Vert_{L^5L^{10}}\Big).
\end{align*}
We apply the Gr\"onwall-type lemma of \cite[Lemma 8.1]{MR2838120}, with 
\begin{gather*}
\varphi=\Vert u-\tilde u\Vert_{L^{10}},\;\gamma=5,\;f=C\Vert u\Vert_{L^{10}}^{4},\;\beta =1, \\
\eta=C\big(\Vert u-\tilde u\Vert_{L^{5}((-T,T),L^{10})}^{5}+\Vert e\Vert_{L^{1}L^{2}}+\Vert S_N(\cdot)((u_0, u_1)-(\tilde u_0, \tilde u_1))\Vert_{L^5L^{10}}\big).
\end{gather*}
We obtain, for all $T>0$ 
\begin{multline*}
\Vert u-\tilde u\Vert_{L^{5}((-T,T),L^{10})}  \leq\Big(\Vert e\Vert_{L^{1}L^{2}}+\Vert    S_N(\cdot)((u_0, u_1)-(\tilde u_0, \tilde u_1)) \Vert_{L^5L^{10}}\\  +\Vert u-\tilde u\Vert_{L^{5}((-T,T),L^{10})}^{5}\Big)
\times\Phi(CM^{4}),
\end{multline*}
where $\Phi(s)=2\Gamma(3+2s)$, $\Gamma$ being the Gamma function. Let $C_M:=6\Phi(CM^4)$ and $\epsilon(M)>0$ be sufficiently small so that, for any $\epsilon\leq\epsilon(M)$
\[
\epsilon^{5}C_M^5\leq \epsilon,\ \text{i.e. }\epsilon\leq 1/C_M^{5/4}.
\]
Then, given $T>0$ so that $\Vert u-\tilde u\Vert_{L^{5}((-T,T),L^{10})}\leq C_M\epsilon$,
we have 
\[
\Vert u-\tilde u\Vert_{L^{5}((-T,T),L^{10})}\leq \Phi(CM^4)\left(2\epsilon+C_M^5\epsilon^5\right),
\]
and thus $\Vert u-\tilde u\Vert_{L^{5}((-T,T),L^{10})}\leq 3 \Phi(CM^4) \epsilon\leq \frac 12 C_M\epsilon$. It
easily follows that we can make $T$ goes to infinity, thus $\Vert u-\tilde u\Vert_{L^{5}(\mathbb{R},L^{10})}\leq \frac 12 C_M\epsilon$
and the lemma follows. The same proof works for the problem in $\mathbb{R}^{3}$
using the corresponding Strichartz estimates.
\end{proof}

\section{Comparison between Neumann and $\mathbb{R}^{3}$ evolutions for dilating
profiles}
\label{section:dilating}

Let us introduce the following notation for the scaling associated to the equation
\begin{defn}
For $\lambda>0$, $\sigma_{\lambda}$ denotes the rescaling\ednote{Thomas: on a parfois besoin de rescaler juste le premier terme. J'ai utilis\'e la m\^eme notation, je ne pense pas qu'il y ait de risque de confusion} on $\dot{H}^1(\RR^3)$, given by
$$\sigma_{\lambda}(f)=\frac{1}{\lambda^{1/2}}f\left(\frac{\cdot}{\lambda}\right)$$
and 
on $\mathcal H (\mathbb{R}^{3})$
given by 
\[
\sigma_{\lambda}(f,g):=\left( \frac{1}{\lambda^{1/2}}f\left(\frac{\cdot}{\lambda}\right),\frac{1}{\lambda^{3/2}}g\left(\frac{\cdot}{\lambda}\right)\right).
\]
\end{defn}

The aim of this section is to show that a dilating profile ($\lambda \rightarrow \infty$) does not
see the obstacle, in the sense that for such profiles, the associated
Neumann and $\mathbb{R}^{3}$ evolutions are asymptoticaly the same.

\begin{lem}[Comparison of linear evolutions for dilating profiles]
\label{lem:dil_lin}Let ${{\vec\psi}}\in \mathcal H (\mathbb{R}^{3})$\ednote{\color{pgreen}Chang\'e $\varphi$ en $\psi$ pour être coherent avec la nouvelle notation des profils},
$f\in L^1(\mathbb R, L^2(\mathbb R ^3))$ be radial,
$(\lambda_{n})_{n\geq1}$ a sequence of
positive real numbers such that $\lambda_{n}\longrightarrow+\infty$, $(t_{n})_{n\geq 1}$ a sequence of
times, $v$ be the solution in the sense of Duhamel of
\begin{align*}
\partial_t^2v-\Delta v&=f\quad\text{in }\RR^3, \\
\vec{v}_{\restriction t=0}&=\vec{\psi},
\end{align*}
and $v_n := \sigma_{\lambda_{n}} v$.
Finally, { let
$f_{n}:=\frac{1}{\lambda_{n}^{\frac{5}{2}}}f(\frac{\cdot-t_{n}}{\lambda_{n}},\frac{\cdot}{\lambda_{n}})$ and}
$u_n$ be the solution in the sense of Duhamel of
\begin{align*}
\partial_t^2u_n-\Delta u_n&=f_n\quad\text{in }B^c, \\
\partial_ru_n&=0\quad \text{for }r=1,\\
\vec{u}_{n\restriction t= -t_n}&=\vec{v}_{n\restriction t= -t_n}.
\end{align*}

Then, as $n$ goes to infinity 
\begin{equation}
\sup_{t\in\mathbb R}\Vert u_{n}{(t)}-v_{n}{(t)}\Vert_{\mathcal H (B^c)}\longrightarrow0,\label{eq:lin_dil_H}
\end{equation}
and  
\begin{equation}
\Vert u_{n}-v_{n}\Vert_{L^{5}L^{10}}\longrightarrow0.\label{eq:lin_dil_L5L10}
\end{equation}
\end{lem}

\begin{proof}
Observe that, by interpolation, it suffices to obtain (\ref{eq:lin_dil_H}): indeed, if   (\ref{eq:lin_dil_H}) holds, by Sobolev embedding we have $\Vert u_{n}-v_{n}\Vert_{L^\infty L^6}\longrightarrow0$, and then (\ref{eq:lin_dil_L5L10}) follows by Hölder inequality, Minkowski inequality, Strichartz estimates for both flows (Propositions \ref{prop:Strichartz} and and \ref{prop:Strichartz_free}) and conservation of energy. Moreover, arguing by density, we can assume that $\vec{\psi}$  and $f$
are smooth and compactly supported. We will argue in three steps:
\begin{enumerate}
\item $t_n = 0$ $\forall n$ and $f=0$ ,
\item $t_n = 0$ $\forall n$ and $\vec \psi = \vec 0$,
\item general case.
\end{enumerate}
\subsection*{Step 1: $t_n = 0$ and $f=0$}
We have, using the equations satisfied by $u_n$ and $v_n$
\begin{multline} \label{eq:lin_dil_dt}
\frac{d}{dt}\Big( \frac{1}{2} \int_{B^c} |\nabla(u_n-v_n)|^2 + \frac{1}{2} \int_{B^c} |\partial_t(u_n-v_n)|^2 \Big) \\
= - \int_{\partial B^c} \partial_r(u_n-v_n) \partial_t (u_n-v_n) =  \int_{\partial B^c} \partial_r v_n \partial_t (u_n-v_n) \\
= 4 \pi \partial_r v_n(t,1) \partial_t (u_n-v_n)(t,1).
\end{multline}
We now claim that, for large $n$
\begin{align}
\label{eq:lin_dil_c1}
|\partial_t v_n (t,1)| + |\partial_r v_n (t,1)|&\lesssim \frac{1}{\lambda_n^{\frac{3}{2}}} \indic_{[-C\lambda_n, C\lambda_n] }, \\ 
\label{eq:lin_dil_c2}
|\partial_t u_n (t,1)| &\lesssim \frac{1}{\lambda_n^{\frac{3}{2}}}+ \frac{e^{-|t|}}{\lambda_n^{\frac{1}{2}}},\end{align}
where the constant $C>0$ and the implicit constants depend on $\vec{\varphi}$. Observe that integrating (\ref{eq:lin_dil_dt}), (\ref{eq:lin_dil_c1}) and (\ref{eq:lin_dil_c2}) give (\ref{eq:lin_dil_H}).

Let us first show (\ref{eq:lin_dil_c1}). Observe that
$$
v_n(t,x)  = \frac{1}{\lambda_n ^ \frac{1}{2}} v(\frac{t}{\lambda_n}, \frac{x}{\lambda_n}),
$$
where $v := S_{\mathbb{R}^{3}}(t)\vec{\psi}$. As $\vec{\psi} \in C^\infty _c$, $\vec v$ is bounded
in any Sobolev space $H^\sigma(\mathbb R ^3)\times H^{\sigma - 1}(\mathbb R ^3)$ for $\sigma \geq 1$. As a consequence,
\begin{equation} \label{eq:lin_dil_c11}
|\partial_t v_n (t,1)| + |\partial_r v_n (t,1)| \lesssim \frac{1}{\lambda_n^{\frac{3}{2}}}.
\end{equation}
Furthermore, by the strong Huygens principle, $v$ is supported in $\big\{ |t| \leq |x| + C \big\}$, and thus 
\begin{equation} \label{eq:lin_dil_c12}
v_n(t,1) = 0 \text{ for }|t| \geq 1 + C \lambda_n.
\end{equation}
Together with (\ref{eq:lin_dil_c11}),  (\ref{eq:lin_dil_c12}) gives  (\ref{eq:lin_dil_c1}).

We now show (\ref{eq:lin_dil_c2}). By Proposition \ref{prop:linear_group}, we have for $t\geq0$
\begin{equation} \label{eq:lin_dil_c21}
\partial_t u_n (t,1) = -\varphi'_{+,n} (1-t) + \varphi'_{-,n}(1+t),
\end{equation}
where, denoting $\vec \psi = (\psi_0, \psi_1)$
\begin{equation}\label{eq:lin_dil_c22}
\varphi'_{-,n}(s) = \frac{1}{2} \Big(  \frac{1}{\lambda_n^{\frac{3}{2}}} \psi_0 ' (\frac{s}{\lambda_n}) +   \frac{1}{\lambda_n^{\frac{3}{2}}} \psi_1 ' (\frac{s}{\lambda_n}) \Big)
\end{equation}
and
\begin{multline} \label{eq:lin_dil_c23}
\varphi'_{+,n}(s) =  - \frac{1}{2} \Big(  \frac{1}{\lambda_n^{\frac{3}{2}}} \psi_0 ' (\frac{2-s}{\lambda_n}) +   \frac{1}{\lambda_n^{\frac{3}{2}}} \psi_1 ' (\frac{2-s}{\lambda_n}) \Big) \\
+ \frac{1}{\lambda_n^{\frac{3}{2}}} \int_1 ^ {2-s} e^{s+\sigma-2} \Big(\psi_0 ' (\frac{\sigma}{\lambda_n}) +  \psi_1 ' (\frac{\sigma}{\lambda_n}) \Big) \, d \sigma 
+ e^{s-1} \frac{1}{\lambda_n^{\frac{1}{2}}} \psi_0(\frac{1}{\lambda_n}).
\end{multline}
This last identity (\ref{eq:lin_dil_c23}) with  (\ref{eq:lin_dil_c21}) and (\ref{eq:lin_dil_c22}) gives (\ref{eq:lin_dil_c2}) for $t\geq 0$. The argument for $t\leq 0$ is similar and Step 1 follows.

\subsection*{Step 2: $t_n = 0$ and ${\vec \psi} = {\vec 0}$ }
As in the first step, we have
\begin{multline} \label{eq:lin_dilinh_dt}
\frac{d}{dt}\Big( \frac{1}{2} \int_{B^c} |\nabla(u_n-v_n)|^2 + \frac{1}{2} \int_{B^c} |\partial_t(u_n-v_n)|^2 \Big) \\
= 4 \pi \partial_r v_n(t,1) \partial_t (u_n-v_n)(t,1).
\end{multline}
Let us show that
\begin{align}
|\partial_t v_n (t,1)| + |\partial_r v_n (t,1)|&\lesssim \frac{1}{\lambda_n^{\frac{3}{2}}} \indic_{[-C\lambda_n, C\lambda_n] }, \label{eq:lin_dilinh_c1}\\
|\partial_t u_n (t,1)| &\lesssim \frac{1}{\lambda_n^{\frac{7}{2}}} t^2,\label{eq:lin_dilinh_c2}
\end{align}
which, together with (\ref{eq:lin_dilinh_dt}), implies (\ref{eq:lin_dil_H}).

We first show (\ref{eq:lin_dilinh_c1}). We have
$$
v_n(t,x)  = \frac{1}{\lambda_n ^ \frac{1}{2}} v(\frac{t}{\lambda_n}, \frac{x}{\lambda_n}),
$$
where $v := S_{\mathbb{R}^{3}}(t)\vec{\psi}$. As $\partial_t v$ and $\partial_r v$ are bounded,
\begin{equation} \label{eq:lin_dilinh_c11}
|\partial_t v_n (t,1)| + |\partial_r v_n (t,1)| \lesssim \frac{1}{\lambda_n^{\frac{3}{2}}}.
\end{equation}
In addition, as we assumed $f$ to be compactly supported in time and space, 
\begin{equation*}
v_n(t,1) = 0 \text{ for }|t| \geq 1 + C \lambda_n,
\end{equation*}
which, with (\ref{eq:lin_dilinh_c11}), gives (\ref{eq:lin_dilinh_c1}).

In order to prove  (\ref{eq:lin_dilinh_c2}), we will need
\begin{claim*}
Let $f\in C^0(\mathbb R \times B^c)$ be radial and bounded:
$$
\forall(t,x)\in\mathbb R \times B^c, \, |f(t,x)|\leq M,
$$
and $w$ be the solution of
\begin{align*}
\partial_t^2w-\Delta w&=f\quad\text{in }B^c, \\
\partial_rw&=0\quad \text{for }r=1,\\
\vec{w}_{\restriction t=0}&= \vec 0.
\end{align*}
Then we have
$$
\forall(t,x)\in\mathbb R \times B^c, \, |w(t,x)| \leq \frac{1}{2} Mt^2.
$$
\end{claim*}
\noindent To obtain (\ref{eq:lin_dilinh_c2}) from the claim, we apply it
to $w := \partial_t u_n$, observing that as $u_n$ is a regular solution,
$\partial_t u_n$ is in $C^0 (\mathbb R, D(-\Delta_N))${, and thus satisfies Neumann boundary conditions}. Let us now
prove the claim to achieve the proof of Step 2. Let
$$
z(t,r) := \frac{1}{2} Mt^2 - w(t,r).
$$
By the formulas
of Proposition \ref{prop:linear_group}, we see that
if $u_1$ is positive for $t\geq 0$, then so is $S_N(0,u_1)(t,r)$. Thus,
by the Duhamel formula, as $(\partial_t ^2 -\Delta) z \geq 0$, we have
$z\geq 0$ for $t\geq 0$, from which we obtain 
$w\leq \frac{1}{2} Mt^2$ for $t\geq 0$.
Considering  $\tilde z(,r) := \frac{1}{2} Mt^2 + w(t,r)$, we obtain
as well $-w\leq \frac{1}{2} Mt^2$ for $t\geq 0$. The negative times
are obtained in a similar fashion.
\subsection*{Step 3: general case}
By the two first steps, we obtain the case $t_n = 0$. Now, let $w_n$ be solution of
the Neumann problem {with initial condition at $t=0$}
\begin{align*}
\partial_t^2w_n-\Delta w_n&=f_n\quad\text{in }B^c, \\
\partial_rw_n&=0\quad \text{for }r=1,\\
\vec{w}_{n\restriction t= {0}}&=\vec{v}_{n\restriction t= {0}}.
\end{align*}
By the case $t_n = 0$, we have, as $n\longrightarrow \infty$
\begin{equation}
\sup_{t\in\mathbb R}\Vert w_{n}-v_{n}\Vert_{\mathcal H (B^c)}\longrightarrow0, \label{eq:lin_s3_1}
\end{equation}
and in particular, as by definition $u_n(-t_n) = v_n(-t_n)$
\begin{equation}
\Vert w_{n}(-t_n)-u_{n}(-t_n)\Vert_{\mathcal H (B^c)}\longrightarrow0.\label{eq:lin_s3_2}
\end{equation}
From (\ref{eq:lin_s3_2}), as $w_{n}-u_n$ is solution
of the homogeneous linear wave equation with Neumann boundary conditions in $B^c$,
it follows from conservation of energy that
\begin{equation}
\sup_{t\in\mathbb R}\Vert w_{n}-u_{n}\Vert_{\mathcal H (B^c)} = \Vert w_{n}(-t_n)-u_{n}(-t_n)\Vert_{\mathcal H (B^c)}\longrightarrow0.\label{eq:lin_s3_3}
\end{equation}
The result (\ref{eq:lin_dil_H}) follows from (\ref{eq:lin_s3_1}) and (\ref{eq:lin_s3_3}).
\end{proof}

The following lemma will play a key role in the comparison between the $\mathbb R^3$ and Neumann dynamics in the
nonlinear profile decomposition introduced in section \ref{section:critical} (see in particular (\ref{eq:comp_U_V})).

\begin{lem}[Comparison of nonlinear evolutions
for dilating profiles] \ednote{\color{pgreen}$U$ et $V$ ont été échangés pour la coherence globale des notations, ce que je n'ai pas indiqu\'e par des couleurs}\. \newline
\label{lem:nl_dil}Let $V\in L^{5}(\mathbb R, L^{10}(\mathbb{R}^3))$
be a solution of the critical defocusing nonlinear wave equation in $\mathbb{R}^{3}$, (i.e.\,(\ref{NLW}) with $\Omega = \mathbb R ^3$ and $\iota = 1$), $(\lambda_{n})_{n}$
a\ednote{Thomas: puisqu'on va extraire des suites \`a tour de bras, il vaut mieux ne pas expliciter l'ensemble de d\'efinition de l'indice $n$ dans les suites. J'ai essay\'e de remplacer syst\'ematiquement $n\geq 1$ o\`u $n\geq 0$ par $n$} sequence of positive real numbers such that $\lambda_{n}\longrightarrow+\infty$,
and $(t_{n})_{n}\in\mathbb{R}^{\mathbb{N}}$. We denote\ednote{David: Tu as pris $\tau_n = 0$; dans section \ref{section:critical}, (\ref{eq:comp_U_V}) on
a besoin de $\tau_n = t_n$, en pratique avoir une translation en temps ici ne change rien donc j'en ai mis une quelconque. Thomas: finalement j'ai pris $t_n=0$ et $\tau_n$ quelconque, puis j'ai appel\'e $\tau_n$: $t_n$. Je n'ai pas marqu\'e tous les changements de ce type par des couleurs!}
$$ V_n(t,x):=\frac{1}{\lambda_n^{1/2}}V\left( \frac{t-t_n}{\lambda_n},\frac{x}{\lambda_n} \right)=\mathscr{S}_{\RR^3}(t)\sigma_{\lambda_n}\left( \vec{V}\left( \frac{-t_n}{\lambda_n} \right) \right)$$
and let $U_n$ be the solution of the nonlinear Neumann problem\ednote{Thomas:J'ai rajout\'e une belle accolade, ainsi que dans les \'equations suivantes} 
\begin{equation*}
\left\{
\begin{aligned}
\partial_t^2U_n-\Delta U_n + U_n^5&=0\quad\text{in }B^c, \\
\partial_rU_n&=0\quad \text{for }r=1,\\
\vec{U}_{n\restriction t= 0}&=\vec{V}_{n\restriction t= 0}.
\end{aligned}\right.
\end{equation*}
Then
$$
\limsup_{n\in\mathbb N} \Vert U_n \Vert_{L^5 L^{10}} < \infty,
$$
and, as $n\longrightarrow \infty$
$$
\sup_{t\in\mathbb R}\big\Vert \vec{U}_{n}(t)-\vec{V}_{n}(t)\big\Vert_{\mathcal H (B^c)}+\Vert U_n - V_n \Vert_{L^5 L^{10}}\longrightarrow 0.
$$
\end{lem}
{
\begin{remark}
The conclusion of the proposition implies
$$  \lim_{n\to\infty}\left\|\vec{\mathscr{S}}_{\RR^3}(t)\sigma_{\lambda_n}\left( \vec{V}\left( \frac{-t_n}{\lambda_n} \right) \right)-\vec{\mathscr{S}}_{N}(t)\sigma_{\lambda_n}\left( \vec{V}\left( \frac{-t_n}{\lambda_n} \right) \right)\right\|_{\HHH (B^c)}=0.$$
\end{remark}
}
\begin{proof}
{Observe that, by energy estimates, it suffices to show $\Vert U_n - V_n \Vert_{L^5 L^{10}}\longrightarrow 0$.} Let $Z_n$ be the solution of the nonlinear Neumann problem
\begin{equation*}
\left\{
\begin{aligned}
\partial_t^2Z_n-\Delta Z_n + V_n^5&=0\quad\text{in }B^c, \\
\partial_rZ_n&=0\quad \text{for }r=1,\\
\vec{Z}_{n\restriction t= 0}&=\vec{U}_{n\restriction t= 0}.\end{aligned}\right.
\end{equation*}
By Lemma\ednote{Thomas: j'ai mis des majuscules partout pour les appels \`a des lemmes, propositions, th\'eor\`emes num\'erot\'es...} \ref{lem:dil_lin} applied to $Z_n(\cdot+t_n)$ and $V_n(\cdot+t_n)$, we get
\begin{equation} \label{eq:dilnl_ZU}
\Vert Z_n - V_n\Vert_{L^5 L^{10}} \longrightarrow 0.
\end{equation}
Let $T>0$\ednote{Thomas: ce n'est plus la peine de d\'efinir ces fonctions, puisque l'ancien $t_n$ est nul. J'ai enlev\'e les tildes partout (sans le noter par des couleurs)}
and observe that
\begin{equation*}
\left\{
\begin{aligned}
\partial_t^2(Z_n-U_n) + \Delta (Z_n-U_n) &= U_n ^5 - V_n ^5 \quad\text{in }B^c, \\
\partial_r(Z_n-U_n)&=0\quad \text{for }r=1,\\
\vec{Z}_{n} - \vec{U}_{n\restriction t= 0}&=\vec{0},
\end{aligned}\right.
\end{equation*}
and therefore, we have, by the global Strichartz estimates for the Neumann flow (Proposition \ref{prop:Strichartz}), together with Hölder and Minkowski inequalities, with an
implicit constant which is independent of $T>0$
\begin{multline}\label{eq:dilnl_dec1} 
\Vert Z_n - U_n \Vert_{L^5(-T,T) L^{10}} \lesssim \Vert U_n^5 - V_n^5 \Vert_{L^1(-T,T) L^{2}} \\
\lesssim \int_{-T} ^T \Big[ \Vert V_n(t) \Vert^4_{L^{10}} \Vert U_n(t) - V_n(t) \Vert_{L^{10}} + \Vert U_n(t) - V_n(t) \Vert^5_{L^{10}} \Big] \; dt \\
\lesssim \int_{-T}  ^T   \Big[ \Vert V_n(t) \Vert_{L^{10}}^4 \Vert Z_n(t) - U_n(t) \Vert_{L^{10}} + \Vert Z_n(t) - U_n(t) \Vert^5_{L^{10}} \Big] \,dt\;+\epsilon_n(T), 
\end{multline}
where we decomposed $U_n(t) - V_n(t) = U_n(t) - Z_n(t) + Z_n(t) - V_n(t)$ in the last line, and 
\begin{equation*}
\epsilon_n(T) =  \int_{-T}^T \Vert V_n(t) - Z_n(t)\Vert_{L^{10}}^5 + \Vert V_n(t)\Vert_{L^{10}}^4\Vert V_n(t) - Z_n(t)\Vert_{L^{10}} \; dt.
\end{equation*}
By Hölder inequality and (\ref{eq:dilnl_ZU})
\begin{equation}\label{eq:dilnl_epsp}
\epsilon'_n := \sup_{T>0} \epsilon_n(T) \leq \Vert {V}_n - {Z}_n\Vert_{L^5(\mathbb R, L^{10})}^5 + \Vert V \Vert_{L^5(\mathbb R, L^{10}(\mathbb R ^3))}^4\Vert {V}_n - {Z}_n\Vert_{L^5(\mathbb R, L^{10})} \longrightarrow 0.
\end{equation}
By (\ref{eq:dilnl_dec1}), we have, with an implicit constant independent of $T$
\begin{multline} \label{eq:dilnl_decprefin}
\Vert Z_n - U_n \Vert_{L^5(-T,T) L^{10}} \lesssim \int_{-T} ^T  \|V_n(t)\|_{L^{10}}^4 \Vert Z_n(t) - U_n(t) \Vert_{L^{10}} \; dt \\ +   \epsilon'_n  + \Vert Z_n - U_n \Vert_{L^5(-T,T) L^{10}} ^5.
\end{multline}
Now, 
$\|V_n\|_{L^{10}}^{4}\in L^{\frac{5}{4}}(\RR)$ and $\left\|\|V_n\|_{L^{10}}^4\right\|_{L^{\frac{5}{4}}(\RR)}=\|V\|_{L^{5}L^{10}}^4.  $
Thus we get,  by (\ref{eq:dilnl_decprefin}), using the Gronwall-type lemma of \cite[Lemma 8.1]{MR2838120}, for all $T>0$, with $C>0$ independent of $T>0$:
\begin{equation} \label{eq:dilnl_decfin}
\Vert Z_n - U_n \Vert_{L^5(-T,T) L^{10}} \leq C \big(  \epsilon'_n  + \Vert Z_n - U_n \Vert_{L^5(-T,T) L^{10}} ^5 \big).
\end{equation} 
Let $\epsilon>0$ be small enough so that 
$
2 C\epsilon^5 \leq \frac{1}{2} \epsilon,
$
and $n$ large enough so that 
$
\epsilon'_n \leq \epsilon ^5.
$
From (\ref{eq:dilnl_decfin}), it follows that if $T$ is such that $\Vert Z_n - U_n \Vert_{L^5(-T,T) L^{10}} \leq \epsilon$, we have
$$
\Vert Z_n - U_n \Vert_{L^5(-T,T) L^{10}} \leq \frac{1}{2} \epsilon.
$$
We can therefore send $T$ to infinity to obtain:
$$
 \Vert Z_n - U_n \Vert_{L^5(\mathbb R, L^{10})}  \longrightarrow 0,
$$ 
and the lemma
 follows using $(\ref{eq:dilnl_ZU})$. 
\end{proof}

\section{Linear profile decomposition}
\label{section:profiles}
{ We recall that by convention, if $(u_0,u_1)\in \HHH(\RR^3)$, $S_N(t)(u_0,u_1)$ (respectively $\mathscr{S}_N(t)(u_0,u_1)$) denotes the flow of the linear (respectively nonlinear) wave equation with Neumann boundary condition applied to the \emph{restriction} of $(u_0,u_1)$ to $B^c$.}
The aim of this section is to show
\begin{prop}[Linear profile decomposition]\label{prop:linprof} Let $(\vec{\phi}_{n})_{n\geq1}$
be a boun\-ded sequence in $\mathcal{H}(B^c)$. Then, up to a subsequence,
there exists sequences of real parameters $(t_{j,n})_{j,n\geq1}$,
$(\lambda_{j,n})_{j,n\geq1}$ and a sequence 
$(\vec{\psi}^{j})_{j\geq1}$
of elements of $\mathcal H (\mathbb{R}^{3})$
such that 
\begin{equation}
j\neq k\implies\lim_{n\rightarrow+\infty}\frac{|t_{j,n}-t_{k,n}|}{\lambda_{j,n}}+\big|\log\frac{\lambda_{j,n}}{\lambda_{k,n}}\big|=+\infty,\label{eq:pl_orth}
\end{equation}
there exists a partition $(J_{\text{comp}},J_{\text{diff}})$ of $\mathbb{N}$
such that
\begin{equation}
j\in J_{\text{comp}}\implies\forall n,\ \lambda_{j,n}=1,\label{eq:pl_lqa0}
\end{equation}
\begin{equation}
j\in J_{\text{diff}}\implies\lambda_{j,n}\underset{n\rightarrow\infty}{\longrightarrow}+\infty,\label{eq:pl_la_inf}
\end{equation}
moreover
\begin{equation}
\forall j,\ t_{j,n}{/\lambda_{j,n}}\longrightarrow\pm\infty\text{ or }\forall n,\;t_{j,n}=0,\label{eq:t_fin}
\end{equation}
and, for all $J\geq1$, 
\begin{equation}
\vec{\phi}_{n}=\sum_{j=1}^{J}\vec{S}_{N}(-t_{j,n})\sigma_{\lambda_{j,n}}\vec{\psi}^{j}+\vec{w}_{n}^{J},\label{eq:pl_dec}
\end{equation}
where the remainder enjoys the decay\ednote{\color{pgreen} Voir l'EdNote \ref{ed:Stri_interp} dans le lemme \ref{lem:nl_dil} -- la meme chose s'applique ici}
\begin{equation}
\hspace{1em}\lim_{J\rightarrow+\infty}\limsup_{n\rightarrow+\infty}\Vert S_{N}(\cdot)\vec{w}_{n}^{J}\Vert_{L^{5}L^{10}}=0.\label{eq:pl_dec_rem}
\end{equation}
In addition, this decomposition verifies the Pythagorean expansion,
\begin{equation}
\forall J,\hspace{1em}\Vert\vec{\phi}_{n}\Vert_{\mathcal{H}(B^c)}^{2}=\sum_{\substack{j\in J_{\text{comp}}\\ 1\leq j\leq J}}\Vert \vec{\psi}^{j}\Vert_{\mathcal{H}(B^c)}^{2}\ +\sum_{
\substack{j\in J_{\text{diff}}\\ 1\leq j\leq J}}\Vert\vec{\psi}^{j}\Vert_{\mathcal{H}(\mathbb{R}^{3})}^{2}\ +\Vert\vec{w}_{n}^{J}\Vert_{\mathcal{H}(B^c)}^{2}\ +o_{n}(1),\label{eq:pl_pyth_1}
\end{equation}
as well as the $L^{6}$ version of it:
\begin{equation}
\forall J,\hspace{1em}\Vert \phi_{n}\Vert_{L^6}^{6}=\sum_{j=1}^{J}\Vert S_{N}(-t_{j,n})\sigma_{\lambda_{j,n}}\vec{\psi}^{j}\Vert_{L^{6}}^{6}\ +\Vert w_{n}^{J}\Vert_{L^{6}}^{6}\ +o_{n}(1).\label{eq:pl_pyth_L6}
\end{equation}
\end{prop}
Recall from \eqref{def:prolong} the definition of the extension operator $\mathcal{P}$.
Proposition \ref{prop:linprof} will be a consequence of:
\begin{lem}
\label{lem:el_conc_0}Let $(f_{n})_{n\geq1}$ be a bounded sequence in
$\dot{H}^1(B^c)$ such that for all sequence of real numbers $(\lambda_{n})_{n\geq1}$
verifying  
$$\lim_n \lambda_n=+\infty \quad \text{or} \quad\forall n,\; \lambda_n=1,$$
we have, as $n$ goes to infinity
\[
\lambda_{n}^{\frac{1}{2}}\mathcal{P}(f_{n})(\lambda_{n}\cdot)\rightharpoonup0\text{ in }\dot{H}^{1}(\mathbb{R}^{3}).
\]
Then, up to a subsequence, as $n$ goes to infinity 
\[
\Vert f_{n}\Vert_{L^{6}(B^c)}\longrightarrow0.
\]
\end{lem}

\begin{proof}
As $(\mathcal{P}(f_{n}))_{n\geq1}$ is a bounded sequence in $\dot{H}_{\text{rad}}^{1}(\mathbb{R}^{3})$,
we may apply the elliptic profile decomposition of \cite{Gerard98}, and up to a subsequence 
\[
\mathcal{P}(f_{n})=\sum_{j=1}^{J}\frac{1}{\lambda_{j,n}^{1/2}}\varphi_{j}\left(\frac{\cdot}{\lambda_{j,n}}\right)+w_{n}^{J},
\]
with 
\[
\lim_{J\rightarrow+\infty}\limsup_{n\rightarrow+\infty}\Vert w_{n}^{J}\Vert_{L^{6}}=0.
\]
Remark that 
\[
\varphi_{j}=\weaklim_{n\rightarrow\infty}\lambda_{j,n}^{1/2}\mathcal{P}(f_{n})(\lambda_{j,n}\cdot)\text{ in }\dot{H}^{1}(\mathbb{R}^{3}).
\]
Thus, for all $j$ such that $\liminf_{n}\lambda_{j,n}>0$, we have $\varphi_{j}=0$
by hypothesis. {Indeed in this case,  extracting subsequences, we can assume that $\lambda_{j,n}$ has a limit $\lambda_{\infty}\in (0,\infty)\cup\{+\infty\}$. If this limit is finite, we may furthermore assume, rescaling $\varphi_j$ if necessary, that $\lambda_{j,n}=1$ for all $n$.}

On the other hand, if $j$ is such that $\lambda_{j,n}\underset{n\rightarrow\infty}{\longrightarrow}0$,
observe that
\[
\lambda_{j,n}^{1/2}\mathcal{P}(f_{n})(\lambda_{j,n}\cdot)=\lambda_{j,n}^{1/2}f_{n}(1)\text{ on }\Big\{ r\leq\frac{1}{\lambda_{j,n}}\Big\}.
\]
By Lemma \ref{lem:u(1)HN},
\[
|f_{n}(1)|\lesssim\Vert f_{n}\Vert_{\dot{H}^1(B^c)}
\]
{which is bounded independently of $n$,} and we deduce that $\lambda_{j,n}^{1/2}\mathcal{P}(f_{n})(\lambda_{n}\cdot)$ goes
to zero as $n$ goes to infinity, uniformly on every compact of $\RR^3$, and thus in the sense
of distributions as well. By the uniqueness of the limit, we conclude
that $\varphi_{j}=0$. Therefore $\varphi_{j}=0$ for all $j$ and
the lemma follows.
\end{proof}
Before showing Proposition \ref{prop:linprof}, let us observe that
\begin{lem}
\label{lem:orth_implies_decay}Let $\big(\vec{R}_n\big)_{n}$ be a sequence in $\HHH(\RR^3)$. For $j=1,2$, let $(\lambda_{j,n})_{n}\in(\mathbb{R}_{+}^{*})^{\mathbb{N}},\ (t_{j,n})_{n}\in\mathbb{R}^{\mathbb{N}}$
be such that\ednote{\color{pgreen}Changé un peu les notations pour être coherent avec celles de la decomposition en profils sans forcement l'indiquer par des couleurs}
\begin{equation}
 \label{alternative_lambda}
\lim_{n\to\infty}\lambda_{j,n}=+\infty \quad\text{or} \quad \forall n,\; \lambda_{j,n}=1.
 \end{equation} 
Then
\begin{enumerate}
\item If 
\[
\exists M,\forall n,\ |t_{1,n}-t_{2,n}|+\big|\log\frac{\lambda_{1,n}}{\lambda_{2,n}}\big|\leq M,
\]
then, up to a subsequence, weakly in $\mathcal{H}(\RR^3)$
\[
\vec{R}_{n}\rightharpoonup0\implies\sigma_{\lambda_{2,n}}^{-1}\vec{\mathcal{P}}(\vec{S}_{N}(t_{1,n}-t_{2,n})\sigma_{\lambda_{1,n}}\vec{R}_{n})\rightharpoonup0.
\]
\item If
\[
\frac{|t_{1,n}-t_{2,n}|}{\lambda_{1,n}}+\big|\log\frac{\lambda_{1,n}}{\lambda_{2,n}}\big|\longrightarrow+\infty,
\]
then, for all $\vec{\psi}\in\mathcal{H}(\RR^3)$, up to a subsequence, weakly in $\mathcal{H}(\mathbb R ^ 3)$:
\[
\sigma_{\lambda_{2,n}}^{-1}\vec{\mathcal{P}}(\vec{S}_{N}(t_{1,n}-t_{2,n})\sigma_{\lambda_{1,n}}\vec{\psi})\rightharpoonup0.
\]
\end{enumerate}
\end{lem}

\begin{proof} 
Let us show the first point. Up to the extraction of a subsequence, we have
\[
t_{1,n}-t_{2,n}\longrightarrow\tau\in\mathbb{R},
\]
and additionally
\[
\text{either }(\lambda_{1,n},\lambda_{2,n})\longrightarrow(+\infty,+\infty),\text{ or }\forall n,\;(\lambda_{1,n},\lambda_{2,n})=(1,1).
\]
In the first situation, Lemma \ref{lem:dil_lin} 
allows us to replace $S_{N}$ by $S_{\mathbb R ^3}$, for which the
result is known.
In the second situation, we have, for any test function
$\vec{\xi}\in\mathcal{H}(B^c)$, 
\[
\sigma_{\lambda_{1,n}}^{-1}\vec{\mathcal{P}}(\vec{S}_{N}(t_{2,n}-t_{1,n})\sigma_{\lambda_{2,n}}\vec{\xi})\longrightarrow\vec{\mathcal{P}}(\vec{S}_{N}(-\tau)\vec{\xi})
\]
strongly in $\mathcal{H}(B^c)$,
and the first point follows. 

Let us now deal with the second point. We are in one of the three following situations:\ednote{Thomas: ici il manquait un cas ($\lambda_{2,n}\to\infty$, $\lambda_{1,n}=1$, $|t_{1,n}-t_{2,n}|\to \infty$), que j'ai incorpor\'e dans le cas (iii). Argument \`a v\'erifier! {\color{pgreen} Je suis d'accord avec l'argument!}}
\begin{enumerate}
\item[(\emph i)] $\lambda_{1,n} \longrightarrow \infty$, 
\item[(\emph{ii})] $\lambda_{2,n} \longrightarrow \infty$, $\forall n,\; \lambda_{1,n}=1$ and $\exists M>0,\;|t_{1,n}-t_{2,n}|  \leq M$,
\item[(\emph{iii})] $\forall n,\; \lambda_{1,n}=1$ and $|t_{1,n}-t_{2,n}|\longrightarrow \infty$.
\end{enumerate}

In the situation (\emph{i}), we can use again Lemma \ref{lem:dil_lin} 
 to replace $S_{N}$ by $S_{\mathbb R ^3}$, and the result follows.
 
 In the situation (\emph{ii}), up to a subsequence, $\vec{\mathcal{P}}(\vec{S}_{N}(t_{1,n}-t_{2,n})\sigma_{\lambda_{1,n}}\vec{\psi})$ is
 converging strongly in $\mathcal H (\mathbb R ^ 3)$ 
  $$
\vec{\mathcal{P}}(\vec{S}_{N}(t_{1,n}-t_{2,n})\sigma_{\lambda_{1,n}}\vec{\psi}) \longrightarrow \vec \xi.
 $$
 By a density argument, we can assume that $\vec \xi$ is smooth and compactly supported. Then,
 by definition of the scaling $\sigma$
 $$
 \forall r \neq 0, \hspace{0.2cm} \sigma^{-1}_{\lambda_{2,n}} \vec \xi (r) \longrightarrow 0
 $$
 and the result follows.
 
 In the situation (\emph{iii}), we use this time Proposition \ref{prop:linear_scattering} to compare the solution to a solution in $\mathbb R^3$, for which the result is known.\ednote{Thomas: ici on a n'a pas besoin que $\lambda_{2,n}$ reste born\'e si je ne m'abuse, ce qui permet de traiter aussi le cas manquant {\color{pgreen} Je suis d'accord!}}
\end{proof}


We are now in position to prove the main result of this section.
\begin{proof}[Proof of Proposition \ref{prop:linprof}]
We will first construct the profiles and the parameters by induction, so that the expansion \eqref{eq:pl_dec} holds together
with the orthogonality of the parameters \eqref{eq:pl_orth}, \eqref{eq:pl_lqa0}, \eqref{eq:pl_la_inf},
and the Pythagorean expansion \eqref{eq:pl_pyth_1}, \eqref{eq:pl_pyth_L6}.
Then, we will show the decay of the remainder \eqref{eq:pl_dec_rem}.

For $\vec \alpha=(\vec {\alpha}_{n})_{n}$ a bounded sequence in $\mathcal{H}(B^c)$,
let us denote by
$\Lambda(\vec \alpha)$ the set of all  $\vec \psi\in \HHH(\RR^3)$ such that there exist an extraction $\{n_k\}_k$ and sequences $(\lambda_{n_k})_k\in (0,\infty)^{\NN}$ and $(t_{n_k})_k\in \RR^{\NN}$, with 
\begin{gather*}
\lim_{k\to\infty}\lambda_{n_k}=\infty\quad \text{or}\quad\forall k,\;\lambda_{n_k}=1,\\
\vec{\psi}=\weaklim_{k\to\infty}\Big(\sigma_{\lambda_{n_{k}}}^{-1}\vec{\mathcal{P}}\big(\vec{S}_{N}(t_{n_{k}})\vec{\alpha}_{n_{k}}\big)\Big)\text{  in }\HHH(\RR^3).
\end{gather*}
We denote\ednote{\label{Ed:Lambda}Thomas: est-ce bien n\'ecessaire de distinguer ces deux cas? Dans le deuxi\`eme cas on a en fait $\|\vec{\psi}\|_{\HHH(B^c)}=\|\vec{\psi}\|_{\HHH(\RR^3)}$ par la d\'efinition de $\mathcal{P}$. On pourrait juste mettre la norme dans $\HHH(\RR^3)$ et \'ecrire une remarque pour dire que celle-ci est \'egale \`a la norme dans $\HHH(B^c)$ dans le cas o\`u $\lambda_{\infty}=1$. Cela permettrait aussi d'all\'eger la d\'efinition de $\Lambda$ puisqu'on n'aurait plus besoin de faire r\'ef\'erence \`a $\lambda_{\infty}$. Qu'en penses-tu? {\color{pgreen} Je suis d'accord, modifications faites!}}
\begin{equation}
\eta(\vec \alpha):=\sup_{\vec \psi \in\Lambda(\vec \alpha)}
\Vert\vec \psi\Vert_{\mathcal{H}(\mathbb{R}^{3})},
 \label{eq:def_eta}
\end{equation} 
{and observe that,  by definition of $\vec{\mathcal P}$ and $\vec \psi$
\begin{equation} \label{eq:prof_same_norms}
\Vert\vec \psi\Vert_{\mathcal{H}(\mathbb{R}^{3})} = \Vert\vec \psi\Vert_{\mathcal{H}(B^c)}\;\text{ if }\vec \psi\text{ is associated with }\lambda_{n_k}=1.
\end{equation} }

\textbf{Extraction of the first profile.} If $\eta((\vec{\phi}_{n})_{n\geq 1})=0$,
then the decomposition holds. Otherwise, there exists $\vec{\psi}^{1}\in\mathcal{H}(\mathbb{R}^{3})$
and $(\lambda_{1,n})_{n\geq1}\in(\mathbb{R}_{+}^{*})^{\mathbb{N}},\ (t_{1,n})_{n\geq1}\in\mathbb{R}^{\mathbb{N}}$
with $\lambda_{1,n}\rightarrow+\infty$ or $\forall n,\;\lambda_{1,n}=1$,
such that, up to an extraction
\begin{equation}
\vec{\psi}^{1}=\weaklim_{{n\to\infty}}\sigma_{\lambda_{1,n}}^{-1}\vec{\mathcal{P}}\big(\vec{S}_{N}(t_{1,n})\vec{\phi}_{n}\big)\text{ in }\ensuremath{\mathcal{H}(\mathbb{R}^{3})},\label{eq:U01}
\end{equation}
and 
$$
\frac{1}{2}\eta((\vec{u}_{n})_{n\geq 1})\leq \Vert\vec{\psi}^{1}\Vert_{\mathcal{H}(\mathbb{R}^{3})}.
$$
 Let us denote
\begin{equation}
\vec{w}_{n}^{1}:=\vec{\phi}_{n}-\vec{S}_{N}(-t_{1,n})\sigma_{\lambda_{1,n}}\vec{\psi}^{1}.\label{eq:wn1_def}
\end{equation}
Observe that, if $t_{1,n}/\lambda_{1,n}$, has a finite limit $\bar{\tau}_1$ \ednote{\color{pgreen}ici je prefère ne pas utiliser $t$ car la limite a l'homogeneite de $t/ \lambda$ et non $t$}, we can harmlessly assume that $t_{1,n}=0$ for all $n$. { Indeed, if $\lambda_{1,n}=1$ for all $n$, we see by \eqref{eq:U01} that 
$$\vec{\mathcal{P}}\left(\vec{S}_{N}(-\bar{\tau}_1) (\vec{\psi}^{1})\right)= \weaklim_{n\to\infty} \vec{\mathcal{P}}\big(\vec{\phi}_{n}\big).$$
If $\lambda_{1,n}\to+\infty$, we have, by \eqref{eq:U01} and Lemma \ref{lem:dil_lin},
\begin{multline*}
\vec{\psi}^{1}=\weaklim_{n\to\infty}\sigma_{\lambda_{1,n}}^{-1}\big(\vec{S}_{\RR^3}(t_{1,n})\vec{\phi}_{n}\big)=\weaklim_{n\to\infty}\left(\vec{S}_{\RR^3}(t_{1,n}/ \lambda_{1,n})\sigma_{\lambda_{1,n}}^{-1}\vec{\phi}_{n}\right)\\
=\weaklim_{n\to\infty}\left(\vec{S}_{\RR^3}(\bar{\tau}_1)\sigma_{\lambda_{1,n}}^{-1}\vec{\phi}_{n}\right).
\end{multline*}
In both cases, we see that we can assume $t_{1,n}=0$ by modifying the limiting profile $\vec{\psi}^1$.
}

Now, we have, by {the} definition { of }$\vec{w}_{n}^{1}$ \eqref{eq:wn1_def} and the weak convergence \eqref{eq:U01}
\begin{multline}
\Big\langle \vec{S}_{N}(-t_{1,n})\sigma_{\lambda_{1,n}}\vec{\psi}^{1},\vec{w}_{n}^{1}\Big\rangle_{\mathcal{H}(B^c)}=\Big\langle\sigma_{\lambda_{1,n}}\vec{\psi}^{1},\vec{S}_{N}(t_{1,n})\vec{\phi}_{n}-\sigma_{\lambda_{1,n}}\vec{\psi}^{1}\Big\rangle_{\mathcal{H}(B^c)}\label{eq:weak_UO1_wn1}\\
=\Big\langle\sigma_{\lambda_{1,n}}\vec{\psi}^{1},\vec{\mathcal{P}}(\vec{S}_{N}(t_{1,n})\vec{\phi}_{n})-\sigma_{\lambda_{1,n}}\vec{\psi}^{1}\Big\rangle_{\mathcal{H}(\RR^3)}=\Big\langle\vec{\psi}^{1},\sigma_{\lambda_{1,n}}^{-1}\vec{\mathcal{P}}(\vec{S}_{N}(t_{1,n})\vec{\phi}_{n})-\vec{\psi}^{1}\Big\rangle_{\mathcal{H}(\RR^3)}\\
\longrightarrow0\text{ as }n\text{ goes to infinity,}
\end{multline}
and therefore, 
\begin{equation}
\Vert\vec{\phi}_{n}\Vert_{\mathcal{H}(B^c)}^{2}=\Vert \vec{S}_{N}(-t_{1,n})\sigma_{\lambda_{1,n}}\vec{\psi}^{1}\Vert_{\mathcal{H}(B^c)}^{2}+\Vert w_{n}^{1}\Vert_{\mathcal{H}(B^c)}^{2}+o_{n}(1).\label{eq:linprof_orth1}
\end{equation}
But, by conservation of energy
\begin{align}
\Vert \vec{S}_{N}(-t_{1,n})\sigma_{\lambda_{1,n}}\vec{\psi}^{1}\Vert_{\mathcal{H}(B^c)}^{2} & =\Vert\sigma_{\lambda_{1,n}}\vec{\psi}^{1}\Vert_{\mathcal{H}(B^c)}^{2}.\label{eq:linprof_orth1_cons}
\end{align}

Now, remark that, if $\lambda_{1,n}\longrightarrow\infty$, then,
as $n$ goes to infinity, we have
\[
\Vert\sigma_{\lambda_{1,n}}\vec{\psi}^{1}\Vert_{\mathcal{H}(B(0,1))}^{2}\longrightarrow0
\]
and thus, as $\sigma_{\lambda_{1,n}}$ is an isometry on $\mathcal H (\mathbb{R}^{3})$,
\begin{align}
\Vert\sigma_{\lambda_{1,n}}\vec{\psi}^{1}\Vert_{\mathcal H{(B^c)}}^{2} & =\Vert\sigma_{\lambda_{1,n}}\vec{\psi}^{1}\Vert_{\mathcal H(\mathbb{R}^{3})}^{2}+o_{n}(1)\nonumber \\
 & =\Vert\vec{\psi}^{1}\Vert_{\mathcal H(\mathbb{R}^{3})}^{2}+o_{n}(1)\hspace{1em}\text{\text{if }\ensuremath{\lambda_{1,n}\longrightarrow\infty}},\label{eq:linprof_lamb_inf_en_scale}
\end{align}
and thus, combining \eqref{eq:linprof_lamb_inf_en_scale} with \eqref{eq:linprof_orth1}
and \eqref{eq:linprof_orth1_cons}, the decomposition \eqref{eq:pl_dec}
with Pythagorean expansion \eqref{eq:pl_pyth_1} holds at rank $J=1$.

Let us now show the $L^{6}$ Pythagorean expansion \eqref{eq:pl_pyth_L6}.

\noindent\emph{First case: $t_{1,n}=0$}. \ednote{\color{pgreen} Ici je crois qu'avec $t_{1,n}=0$ on peut en fait rassembler les deux cas $\lambda_{1,n}=1$ et $\lambda_{1,n}\longrightarrow\infty$, car on a plus de problème de commutation de $S_N({t_{1,n}})$ avec le scaling. (Moralement on l'a fait commuter lorsqu'on a translate les profils), voir ce qui suit}
Let\ednote{Thomas: dans la suite j'ai remplac\'e ${\psi}^{1}$ par $\vec{\psi}^1$ quand il le fallait}
$$
f_n := 
\bigg| \int  |{\phi_{{}n}}|^6-|\sigma_{\lambda_{1,n}}{{\psi}^{1}}|^6-|{w_{n}^{1}}|^6  \bigg|, $$
and observe that, as for any $z,w \in \mathbb R$
$$
\big| |z+w|^6 - |z|^6 - |w|^6 \big| \lesssim |z||w| \big( |z|^4 + |w|^4\big),
$$
we have, by (\ref{eq:wn1_def})
$$
f_n \lesssim \int \left|\sigma_{\lambda_{1,n}}{{\psi}^{1}}\right|\,|w_{n}^{1}|\;g_n, \hspace{0.2cm}g_n := \left|\sigma_{\lambda_{1,n}}{{\psi}^{1}} \right|^4 + |w_{n}^{1}|^4.
$$
On the other hand, by Sobolev embedding, conservation of energy and scale invariance
\begin{equation*}
\left\Vert  \sigma_{\lambda_{1,n}}{\psi}^{1} \right\Vert_{L^6} \lesssim \left\Vert  \sigma_{\lambda_{1,n}}\vec{\psi}^{1} \right\Vert_{\mathcal H (B^c)}  \\ \leq \left\Vert \sigma_{\lambda_{1,n}}\vec{\psi}^{1} \right\Vert_{\mathcal H (\mathbb R ^3)} = \left\Vert {\vec\psi}^{1}\right\Vert_{\mathcal H (\mathbb R ^3)},
\end{equation*}
Together with (\ref{eq:wn1_def}) and Sobolev embedding, it follows that $\sup_n \Vert g_n \Vert_{L^{3/2}} < \infty$,
and we get, by H\"older inequality\ednote{Thomas: j'ai enlev\'e le dernier terme, qui n\'ecessitait de changer le domaine d'int\'egration}
\begin{multline}\label{eq:PL6fn}
f_n \lesssim \Big( \int_{B^c} |\sigma_{\lambda_{1,n}}{{\psi}^{1}}|^3|w_{n}^{1}|^3  \Big)^\frac{1}{3} \\
{
\leq  \Big( \int_{\mathbb R ^ 3} |\sigma_{\lambda_{1,n}}{{\psi}^{1}}|^3|\tilde{w}_n^1|^3  \Big)^\frac{1}{3} =  \Big( \int_{\mathbb R ^ 3} |{{\psi}^{1}}|^3| \sigma_{\lambda_{1,n}^{-1}}\tilde{w}_n^1|^3  \Big)^\frac{1}{3},}
\end{multline}
{where $\vec {\tilde{w}}_n^1 := \vec{\mathcal P}\vec{\phi}_{n}-\sigma_{\lambda_{1,n}}\vec{\psi}^{1}$ extends the definition
of $\vec {w}_n^1$ to $\mathbb R^3$ in the present case $t_{1,n} = 0$.}
Now, observe that by (\ref{eq:U01}) and (\ref{eq:wn1_def}), $\sigma_{\lambda_{1,n}^{-1}}\vec {\tilde{w}}_n^1 \rightharpoonup 0$
weakly in $\mathcal H (\mathbb R^3)$. By Rellich theorem, for any compact $K\subset \mathbb R ^3$, $\sigma_{\lambda_{1,n}^{-1}}\tilde{w}_n^1$ strongly converges to $0$ in $L^4(K)$. It follows that 
$|\sigma_{\lambda_{1,n}^{-1}}\tilde{w}_n^1|^3$ converges strongly to $0$ in $L^{4/3}(K)$.
By Sobolev embedding,  $|\sigma_{\lambda_{1,n}^{-1}}\tilde{w}_n^1|^3$ is
bounded in $L^2(\mathbb R^3)$, thus has a weakly convergent subsequence in $L^2(\mathbb R^3)$. By uniqueness of the limit {in the sense of distributions},
this weak limit is zero and \eqref{eq:pl_pyth_L6} follows from (\ref{eq:PL6fn}).


\noindent\emph{Second case: $t_{1,n}/\lambda_{1,n}\longrightarrow \pm\infty$.} In this case, we have 
$$
 \Vert {S_{N}}(-t_{1,n}/\lambda_{1,n})\vec{\psi}^{1} \Vert_{L^6} \underset{n\to\infty}{\longrightarrow} 0,
 $$
 which can be proved easily  
 from the corresponding property for the free flow $S_{\mathbb{R}^3}$, and Proposition \ref{prop:linear_scattering}. The $L^6$ Pythagorean expansion  \eqref{eq:pl_pyth_L6} follows immediately.

\textbf{Extraction of the subsequent profiles.} Let us show how to
extract the second profile, the extraction of the $J$'th from the
$J-1$'th being the same for arbitrary $J\geq2$. If $\eta(\vec{w}_{n}^{1})=0$,
then we are done, otherwise, there exists $\vec{\psi}^{2}\in\mathcal{H}(\mathbb{R}^{3})$
and $(\lambda_{2,n})_{n\geq1}\in(\mathbb{R}_{+}^{*})^{\mathbb{N}},\ (t_{2,n})_{n\geq1}\in\mathbb{R}^{\mathbb{N}}$
with $\lambda_{2,n}\rightarrow+\infty$ or $\lambda_{2,n} = 1$,
such that
\begin{equation}
\vec{\psi}^{2}=\weaklim\sigma_{\lambda_{2,n}}^{-1}\vec{\mathcal{P}}\big(\vec{S}_{N}(t_{2,n})\vec{w}_{n}^{1}\big)\text{ in }\ensuremath{\mathcal{H}(\mathbb{R}^{3})},\label{eq:U02}
\end{equation}
and
$$
\frac{1}{2}\eta((\vec{w}_{n}^{1})_{n\geq1})\leq \Vert\vec{\psi}^{2}\Vert_{\mathcal{H}(\mathbb{R}^{3})}.
$$
We take
\begin{align}
\vec{w}_{n}^{2} & :=\vec{w}_{n}^{1}-\vec{S}_{N}(-t_{2,n})\sigma_{\lambda_{2,n}}\vec{\psi}^{2}\label{eq:wn2_def}\\
 & \ =\vec{u}_{n}-\vec{S}_{N}(-t_{2,n})\sigma_{\lambda_{2,n}}\vec{\psi}^{2}-\vec{S}_{N}(-t_{1,n})\sigma_{\lambda_{1,n}}\vec{\psi}^{1}.\nonumber 
\end{align}
Let us first show the orthogonality condition \eqref{eq:pl_orth}.
Denoting
\[
\vec{r}_{n}^{1}:=\sigma_{\lambda_{1,n}}^{-1}\vec{\mathcal{P}}\left(\vec{S}_{N}(t_{1,n})\vec{w}_{n}^{1}\right){=\sigma_{\lambda_{1,n}}^{-1}\vec{\mathcal{P}}\left(\vec{S}_{N}(t_{1,n})\vec{u}_{n}\right)-\sigma_{\lambda_{1,n}^{-1}} \vec{\mathcal{P}}\sigma_{\lambda_{1,n}}\vec{\psi}^1},
\]
we have, by \eqref{eq:U01} and \eqref{eq:wn1_def}
\[
\vec{r}_{n}^{1}\rightharpoonup0\text{ weakly in }\mathcal{H}(\mathbb R ^ 3),
\]
and in addition\ednote{Thomas: il me semble qu'on n'utilise pas \eqref{eq:wn2_def} ici. \color{pgreen} en effet!}, {by} \eqref{eq:U02}
\[
\sigma_{\lambda_{2,n}}^{-1}\vec{\mathcal{P}}(\vec{S}_{N}(t_{2,n}-t_{1,n})\sigma_{\lambda_{1,n}}\vec{r}_{n}^{1})\rightharpoonup\vec{\psi}^{2}\neq0,
\]
therefore, by Lemma \ref{lem:orth_implies_decay}, the orthogonality
condition \eqref{eq:pl_orth} for $(j,k)=(1,2)$ follows.

To show the Pythagorean expansion \eqref{eq:pl_pyth_1}, using the
arguments of the case $J=1$, it suffices to show that the newly arising
mixed term goes to zero, namely that
\[
\langle \vec{S}_{N}(-t_{2,n})\sigma_{\lambda_{2,n}}\vec{\psi}^{2},\vec{S}_{N}(-t_{1,n})\sigma_{\lambda_{1,n}}\vec{\psi}^{1}\rangle_{\mathcal{H}(B^c)}\underset{{n\to\infty}}{\longrightarrow0}.
\]
{Noting that the left-hand side of the previous line equals}
\[
\langle\vec{\psi}^{2},\sigma_{\lambda_{2,n}}^{-1}\vec{\mathcal{P}}(\vec{S}_{N}(t_{2,n}-t_{1,n})\sigma_{\lambda_{1,n}}\vec{\psi}^{1})\rangle_{\mathcal{H}(B^c)},
\]
the result follows by the orthogonality condition together with
Lemma \ref{lem:orth_implies_decay}. 

Finally, (\ref{eq:U02}) and (\ref{eq:wn2_def}) imply by the exact same arguments as in the extraction of the first profile 
that
$$
 \Vert w_n^1 \Vert_{L^6} ^ 6 = \Vert {S_{N}}(-t_{2,n})\sigma_{\lambda_{2,n}}\vec{\psi}^{2}  \Vert_{L^6} ^ 6 + \Vert w_n ^ 2\Vert_{L^6}^6+ o_n(1),
$$
from which the $L^6$ Pythagorean expansion \eqref{eq:pl_pyth_L6} follows using the decomposition proved at the previous rank,
which readed 
$$
\Vert \phi_n \Vert_{L^6}^{6} =  \Big\Vert {S_{N}}(-t_{1,n})\sigma_{\lambda_{1,n}}\vec{\psi}^{1}  \Big\Vert_{L^6}^{6} + \left\Vert w_n ^ 1\right\Vert_{L^6}^{6} + o_n(1).
$$

\textbf{Labeling.} We define $J_{\text{diff}}$ and $J_{\text{comp}}$ as follows: if $\lambda_{j,n}=1$ for all $n$,
then $j\in J_{\text{comp}}$, otherwise, $j\in J_{\text{diff}}$.

\textbf{Decay of the remainder. }In order to obtain \eqref{eq:pl_dec_rem},
it suffices to show that
\begin{equation}
\lim_{J\rightarrow+\infty}\limsup_{n\rightarrow+\infty}\Vert S_{N}(\cdot)\vec{w}_{n}^{J}\Vert_{L^{\infty}L^{6}}=0.\label{eq:dec_rem_infty_6}
\end{equation}
Indeed, if \eqref{eq:dec_rem_infty_6} holds, Strichartz estimates
of Proposition \ref{prop:Strichartz} together with Hölder inequality,
conservation of energy, and the fact that, by the Pythagorean expansion
\eqref{eq:pl_pyth_1}, 
\[
\forall J,\ \limsup_{n\rightarrow+\infty}\Vert\vec{w}_{n}^{J}\Vert_{\mathcal{H}(B^c)}\leq\limsup_{n\rightarrow+\infty}\Vert\vec{\phi}_{n}\Vert_{\mathcal{H}(B^c)},
\]
yields \eqref{eq:pl_dec_rem}.

Let us show \eqref{eq:dec_rem_infty_6}. To this purpose, observe
that, by the Pythagorean expansion \eqref{eq:pl_pyth_1}, 
\[
\forall J,\hspace{1em}\sum_{j=1,\ j\in J_{\text{comp}}}^{J}\Vert\vec{\psi}^{j}\Vert_{\mathcal{H}(B^c)}^{2}\ +\sum_{j=1,\ j\in J_{\text{diff}}}^{J}\Vert\vec{\psi}^{j}\Vert_{\mathcal{H}(\mathbb{R}^{3})}^{2}\leq\limsup_{n\geq1}\Vert\vec{\phi}_{n}\Vert_{\mathcal{H}(B^c)}^2,
\]
and thus both series in $j$ are convergent. Because{, by (\ref{eq:prof_same_norms}),} the profiles
are constructed in such a way that\ednote{Thomas: une question typographique cruciale :-)
Que penserais-tu de ${\psi}^{j}$ plut\^ot que ${\psi}^{j}$? {\color{pgreen} J'ai remplace tous les $\overrightarrow{x^a_b}$ par $\vec x ^ a _b$}}
\[
\eta((\vec{w}_{n}^{j})_{n\geq1})\leq2\begin{cases}
\Vert\vec{\psi}^{j}\Vert_{\mathcal{H}(\mathbb{R}^{3})} & \text{if }j\in J_{\text{diff}},\\
\Vert\vec{\psi}^{j}\Vert_{\mathcal{H}(B^c)} & \text{if }j\in J_{\text{comp}},\ j\neq0
\end{cases}
\]
it follows that
\begin{equation}
\eta\Big((\vec{w}_{n}^{J})_{n\geq1}\Big)\underset{J\rightarrow\infty}{\longrightarrow}0.\label{eq:dec_eta}
\end{equation}
Arguing by contradiction, the $L^{\infty}L^{6}$ decay of $S_{N}(\cdot)\vec{w}_{n}^{J}$
follows by Lemma \ref{lem:el_conc_0}: indeed, if the decay of the
remainder \eqref{eq:dec_rem_infty_6} does not hold, by a diagonal
argument, there exists $\epsilon_{0}>0$ and sequences $J_{k}\longrightarrow+\infty$,
$n_{k}\longrightarrow+\infty$, and $t_{k}$ such that\ednote{Thomas: il faut faire attention, l'affirmation que $\eta(...)$ est nulle n'est pas automatique, cela d\'epend du choix des $J_k$ et $n_k$ {\color{pgreen}Merci, j'etais passé à coté!}}
\[
{\eta\Big((\vec{w}_{n_k}^{J_k})_{k}\Big)=0\quad \text{and}}\quad 
\forall k,\ \Big\Vert S_{N}(t_{k})\vec{w}_{n_{k}}^{J_{k}}\Big\Vert_{L^{6}(B^c)}\geq\epsilon_{0}.
\]
Using Lemma \ref{lem:el_conc_0}, it follows that there exists $\vec{\psi}\in \mathcal H (\mathbb{R}^{3})$,
$\vec{\psi}\neq0$, and a sequence {$(\lambda_{k})_k$ with
$$\lim_{k}\lambda_k=\infty\text{ or }\forall k,\;\lambda_k=1$$}
such that, after extraction
\[
\sigma_{\lambda_{k}}^{-1}\vec{\mathcal{P}}(\vec{S}_{N}(t_{k})\vec{w}_{n_{k}}^{J_{k}})\rightharpoonup\vec{\psi}
\]
weakly in $\mathcal H (\mathbb R ^3)$.
This contradicts the definition \eqref{eq:def_eta} of $\eta$ and ends the
proof of the proposition.
\end{proof}

\section{Construction of a compact flow solution}
\label{section:critical}
Let us define the critical energy $E_{c}$ by 
\begin{equation} \label{eq:Ec}
E_{c}:=\sup\big\{ E>0,\hspace{1em}\forall\vec{u}\in\mathcal{H}(B^c),\ \mathscr{E}(\vec{u})\leq E\implies\mathscr{S}_{N}(\cdot)\vec{u}\in L^{5}L^{10}\big\} ,
\end{equation}
where, for $\vec{u}\in\mathcal{H}(B^c)$, $\mathscr{E}$ is as before the conserved energy
\[
\mathscr{E}(\vec{u}):=\frac{1}{2}\Vert\vec{u}\Vert_{\mathcal{H}(B^c)}^{2}+\frac{1}{6}\Vert u\Vert_{L^{6}}^{6}.
\]
Observe that $E_{c}>0$ by Proposition \ref{prop:perturb_scat}. The
aim of this section is to show
\begin{thm} \label{th:conc-comp}
If $E_{c}<+\infty$, then there exists $\vec{u}_{c}\in\mathcal{H}(B^c)$, $\vec{u}_{c} \neq \vec{0}$,
such that the nonlinear flow $\big\{ \vec{\mathscr{S}}_{N}(t)\vec{u}_{c}, \ t\in \mathbb R \big\}$ has a 
compact closure in $\mathcal{H}(B^c)$.
\end{thm}

\begin{proof}
If $E_{c}<+\infty$, let $\vec{u}_{0}^{n}$ be a minimising
sequence for $E_{c}$, in the sense that
\begin{equation}
\mathscr{E}(\vec{u}_{0}^{n})\geq  E_{c},\ \lim_{n\rightarrow\infty}\mathscr{E}(\vec{u}_{0}^{n})=E_{c},\ \mathscr{S}_{N}(\cdot)\vec{u}_{0}^n\notin L^{5}L^{10}.\label{eq:suite_u0n}
\end{equation}
{Translating $u_n=\mathscr{S}_N(\cdot)\vec{u}_0^n$ in time if necessary, we may assume
\begin{equation}
 \label{symL5L10infty}
\lim_{n\to\infty} \|u_n\|_{L^{5}\left((0,+\infty),L^{10}\right)}=\lim_{n\to\infty} \|u_n\|_{L^{5}\left((-\infty,0),L^{10}\right)}=+\infty,
 \end{equation} 
where by convention $\|u_n\|_{L^{5}\left((-\infty,0),L^{10}\right)}=+\infty$ if $u_n\notin L^{5}\left((-\infty,0),L^{10}\right)$, and similary for $L^{5}\left((0,\infty),L^{10}\right)$.
}
As $\vec{u}_{0}^{n}$ is bounded in $\mathcal{H}(B^c)$, we
can, up to a subsequence, decompose it into profiles according to Proposition
\ref{prop:linprof}: 
\begin{equation}
\vec{u}_{0}^{n}=\sum_{j=1}^{J}\vec{S}_{N}(-t_{j,n})\sigma_{\lambda_{j,n}}\vec{\psi}^{j}+\vec{w}_{n}^{J}.\label{eq:dec_u0n}
\end{equation}
\ednote{\color{pgreen}Ici j'ai pris mon courage à deux mains, et pour être coherent avec les notations j'ai chang\'e la definition des profils et appelé $U$ les profils Neumann et $V$ les profils $\mathbb R ^3$. Ce qui conduit à un peu de réorganisation au début de ce qui suit. En conséquence de quoi, par la suite la plupart des $U$ et $V$ sont échangés par rapport a la version précédente. Je n'ai pas indiqué ces modifications par des couleurs.}To each profile $(\vec{\psi}^{j},(\lambda_{j,n})_{n\geq1},(t_{j,n})_{n\geq1})$
we associate a {family of nonlinear Neumann profiles $({U_n^{j}})_{n\geq1}$, and additionally, for
$j\in J_{\text{diff}}$, a free nonlinear profile $V^j$ and its rescaled family $({V_n^{j}})_{n\geq1}$,} in the following way:
\begin{itemize}
\item If $j\in J_{\text{comp }}$ i.e. $\lambda_{j,n}=1$, let ${U^{j}}$ be the only
solution of the critical nonlinear wave equation with Neumann boundary conditions \eqref{NLWrad}--\eqref{NBDrad}, given by Proposition \ref{prop:perturb_scat} such that
\begin{equation}
\lim_{n\rightarrow\infty}\left\Vert{\vec{U}^{j}}({-t_{j,n}})-\vec{S}_{N}(-t_{j,n})\vec{\psi}^{j}\right\Vert_{\mathcal{H}(B^c)}=0,\label{eq:def_Uj_comp}
\end{equation} {
and we set
\begin{equation}
U_{n}^{j}(t):=U^{j}({t-t_{j,n}}).\label{eq:def_Vnj_comp}
\end{equation}
Notice that, if $-t_{j,n}\to \pm\infty$, $U^{j}\in L^{5}(\mathbb{R}_{\pm},L^{10}( B^c))$ by construction.}

\item 
If $j\in J_{\text{diff}},$ i.e. $\lambda_{j,n}\rightarrow\infty$, by Lemma \ref{lem:dil_lin},
$$ \lim_{n\to\infty} \left\|\vec{S}_{N}(-t_{j,n})\sigma_{\lambda_{j,n}}\vec{\psi}^{j} -\vec{S}_{\RR^3}(-t_{j,n})\sigma_{\lambda_{j,n}}\vec{\psi}^{j}\right\|_{\HHH {(B^c)}}=0. $$
Furthermore, denoting by $V_L^j(t):=S_{\RR^3}(t){\psi}^{j}(t)$, we have
$$S_{\RR^3}(t-t_{j,n})\sigma_{\lambda_{j,n}}{\psi}^{j}=\frac{1}{\lambda_{j,n}^{1/2}}V^j_L\left(\frac{t-t_{j,n}}{\lambda_{j,n}},\frac{x}{\lambda_{j,n}}  \right).$$
We define {the free nonlinear profile} $V^j$ as the unique solution of the critical nonlinear wave equation on $\RR^3$ such that if $t_{j,n}=0$ for all $n$, $\vec{V}^{j}(0)={\psi}^{j} $
and if $\lim_{n\to\infty} -t_{j,n}/\lambda_{j,n}=\pm\infty$, $\lim_{t\to \pm\infty} \left\|\vec{V}^{j}(t)-\vec{V}_L^{j}(t)\right\|_{\HHH {(B^c)}}=0$. In other terms:
\begin{equation}
\label{eq:def_Uj_diff}
\lim_{n\to \pm\infty} \left\|\vec{V}^{j}(-t_{j,n}/\lambda_{j,n})-\vec{V}_L^{j}(-t_{j,n}/\lambda_{j,n})\right\|_{\HHH {(B^c)}}=0.
\end{equation}
{ Furthermore, we set
$$
{V}_{n}^{j}(t):=\frac{1}{\lambda_{j,n}^{1/2}}V^{j}\Big(\frac{t-t_{j,n}}{\lambda_{j,n}}\Big),
$$
and we then define the associated family of nonlinear Neumann profiles as
\begin{equation}
U_{n}^{j}(t):=\mathscr{S}_N(t)\left(\vec{V}_n^j(0)\right)=\mathscr{S}_{N}(t)\left(\sigma_{\lambda_{j,n}}
\Big(\vec{V}^{j}\Big(\frac{-t_{j,n}}{\lambda_{j,n}}\Big)\Big)\right).
\label{eq:def_Vnj_diff}
\end{equation} 
Observe that, as a solution
of a defocusing nonlinear wave equation in $\mathbb{R}^{3}$, for
which the scattering is well known, we have $V^j\in L^{5}L^{10}(\mathbb{R}^{3})$.
Furthermore, as $\vec{U}_n^j(0) = \vec{V}_n^j(0)$, Lemma \ref{lem:nl_dil} (used with $t_n= t_{j,n}$) yields
$$
\forall j\in J_{\text {diff}},\hspace{1em}\sup_n \Vert U_n^j\Vert_{L^{5}(\mathbb{R},L^{10}( B^c))} < \infty,
$$
and
\begin{equation}
\forall j\in J_{\text{diff}},\hspace{1em}\sup_{t}\left\|\vec{V}_{n}^{j}(t)-\vec{U}_{n}^{j}(t)\right\|_{\HHH {(B^c)}}+\Vert V_{n}^{j}-U_{n}^{j}\Vert_{L_{t}^{5}L_{x}^{10}}\underset{n\to\infty}{\longrightarrow}0.\label{eq:comp_U_V}
\end{equation} }

\end{itemize}

Let us assume from now on, by contradiction, that the decomposition \eqref{eq:dec_u0n}
has strictly more than one non trivial profile, i.e
\begin{equation}
J>1.\label{eq:more_than_one_profile}
\end{equation}
Then, by the Pythagorean expansion \eqref{eq:pl_pyth_1} together
with its $L^{6}$ version \eqref{eq:pl_pyth_L6}
\[
\forall j\in J_{\text{comp}},\hspace{1em}\limsup_{n\rightarrow\infty}\;\mathscr{E}\left(S_{N}(-t_{j,n})\vec{\psi}^{j}\right)<E_{c}.
\]
Hence, by \eqref{eq:def_Uj_comp}, $\mathscr{E}(U^j)<E_C$,
and ${U^{j}}\in L^{5}L^{10}( B^c)$ by the definition of the critical
energy. Summing up, we have {
\begin{equation}
\forall j\in J_\text{comp},\hspace{0.3em}  U^j\in{L^{5}(\mathbb{R},L^{10}( B^c))};\; \forall j\in J_\text{diff},\hspace{0.3em}  V^j\in{L^{5}(\mathbb{R},L^{10}( B^c))}.\label{eq:tout_est_diffusif}
\end{equation} }

Let $u_{n}:=\mathscr{S}_{N}\vec{u}_{0}^{n}$. We will show
the following nonlinear profile decomposition:
\begin{prop}
\label{prop:non_lin_prof}We have
\begin{align}
\forall J,\ u_{n}(t)  &=\sum_{1\leq j \leq J} U_n^j(t) + R_{n}^{J}(t) \label{eq:dec_uj} \\
&= \sum_{\substack{j\in J_{\text{comp}}\\1\leq j\leq J}} U_n^j(t)+ \sum_{\substack{j\in J_{\text{diff}}\\1\leq j\leq J}} V_n^j(t) + \tilde R_{n}^{J}(t), \nonumber
\end{align}
where
\begin{equation*}
\lim_{J\rightarrow\infty}\limsup_{n\rightarrow\infty}\Vert R_{n}^{J}\Vert_{L^{5}L^{10}}= \lim_{J\rightarrow\infty}\limsup_{n\rightarrow\infty}\Vert \tilde R_{n}^{J}\Vert_{L^{5}L^{10}}=0.
\end{equation*}
\end{prop}

{To this purpose, let\ednote{\color{pgreen}à partir d'ici la plupart des $U$ et $V$ sont échangés par rapport a la version précédente, sans indication de couleur. J'ai aussi changé $v$ en $\tilde u$ puisque c'est une solution de Neumann}}
\begin{equation}
\tilde u_{n}^{J}:=\sum_{j=1}^{J}U_{n}^{j}+z_{n}^{J},\label{eq:def_vnj}
\end{equation}
where 
\begin{equation} \label{eq:def_znJ}
z_{n}^{J}(t):=S_{N}(t)\vec{w}_{n}^{J}, 
\end{equation} 
verifies, by the decay of the remainder of the linear profile
decomposition 
\begin{equation}
\lim_{J\to\infty}\limsup_{n\to\infty}\Vert z_{n}^{J}\Vert_{L^{5}L^{10}}=0.\label{eq:dec_znJ}
\end{equation}

Observe that $\tilde u_{n}^{J}$ is solution in $ B^c$ of the following nonlinear wave
equation with Neumann boundary conditions:
\begin{equation}
(\partial_{t}^{2}-\Delta_{N}){\tilde u}_{n}^{J}+({\tilde u}_{n}^{J})^{5}=e_{n}^{J},\hspace{1em}\text{with }e_{n}^{J}:=({\tilde u}_{n}^{J})^{5}-\sum_{j=1}^{J}(U_{n}^{j})^{5}.\label{eq:eq_vnj}
\end{equation}
Let us show\ednote{Thomas: j'ai chang\'e la proposition suivante en lemme.}
\begin{lem}
We have
\begin{equation}
\lim_{J\to\infty}\limsup_{n\rightarrow\infty}\Vert e_{n}^{J}\Vert_{L^{1}L^{2}}=0,\label{eq:decay_en_J}
\end{equation}
and
\begin{equation}
\vec {\tilde u}_{n\restriction t = 0}^J
=\vec{u}_{n}+\vec{\alpha}_{n}^{J};\hspace{1em}\ \lim_{J\to\infty}\limsup_{n\to\infty}\Vert S_N(\cdot) \alpha_{n}^{J}\Vert_{L^5L^{10}}=0.\label{eq:v_n_data}
\end{equation}
\end{lem}

\begin{proof}
We will first show \eqref{eq:decay_en_J}. We have\ednote{Thomas: $A^4B\leq A^5+A B^4$ par l'in\'egalit\'e de Young.}
\begin{equation}
|e_{n}^{J}| \lesssim_{J}\sum_{1\leq j\neq k\leq J}|U_{n}^{j}|^{4}|U_{n}^{k}|+|z_{n}^{J}|^{5}+|z_{n}^{J}|\sum_{j=1}^{J}|U_{n}^{j}|^{4}.
\end{equation}
Let us begin with the mixed terms $|U_{n}^{j}|^{4}|U_{n}^{k}|$. We
start with the case $j,k\in J_{\text{diff}}$. Notice that
\[
|U_{n}^{j}|^{4}|U_{n}^{k}|\leq|V_{n}^{j}|^{4}|V_{n}^{k}|+|U_{n}^{j}|^{4}|V_{n}^{k}-U_{n}^{k}|+|V_{n}^{j}||V_{n}^{k}-U_{n}^{k}|^{4},
\]
thus we get, by Hölder inequality
\begin{multline}
\left\Vert|U_{n}^{j}|^{4}|U_{n}^{k}|\right\Vert_{L^{1}L^{2}}\leq\left\Vert|V_{n}^{j}|^{4}|V_{n}^{k}|\right\Vert_{L^{1}L^{2}}+\left\Vert U_{n}^{j}\right\Vert_{L^{5}L^{10}}^{4}\left\Vert V_{n}^{k}-U_{n}^{k}\right\Vert_{L^{5}L^{10}}\\
+\left\Vert V_{n}^{j}\right\Vert_{L^{5}L^{10}}\left\Vert V_{n}^{k}-U_{n}^{k}\right\Vert_{L^{5}L^{10}}^{4}.\label{eq:ineq_Vnjk_jk_diff}
\end{multline}
On the one hand, as $V_{n}^{j}$ and $V_{n}^{k}$ are rescaled solutions
of the defocusing critical nonlinear wave equation in $\mathbb{R}^{3}$ associated with orthogonal parameters,
it is well known that, as $n$ goes to infinity (see for example \cite{BahouriGerard99})
\begin{equation}
\left\Vert|V_{n}^{j}|^{4}|V_{n}^{k}|\right\Vert_{L^{1}L^{2}}\longrightarrow0.\label{eq:ineq_Vnjk_jk_diff_1}
\end{equation}
On the other hand, as
\[
\sup_{n}\ \left\Vert U_{n}^{j}\right\Vert_{L^{5}L^{10}}+\left\Vert V_{n}^{j}\right\Vert_{L^{5}L^{10}}<\infty,
\]
it follows from \eqref{eq:comp_U_V} that
\begin{equation}
\left\Vert U_{n}^{j}\right\Vert_{L^{5}L^{10}}^{4}\left\Vert V_{n}^{k}-U_{n}^{k}\right\Vert_{L^{5}L^{10}}+\left\Vert V_{n}^{j}\right\Vert_{L^{5}L^{10}}\left\Vert V_{n}^{k}-U_{n}^{k}\right\Vert_{L^{5}L^{10}}^{4}\longrightarrow0\label{eq:ineq_Vnjk_jk_diff_2}
\end{equation}
as $n$ goes to infinity, and thus \eqref{eq:ineq_Vnjk_jk_diff} combined
with \eqref{eq:ineq_Vnjk_jk_diff_1} and \eqref{eq:ineq_Vnjk_jk_diff_2}
gives
\begin{equation}
\left\Vert|U_{n}^{j}|^{4}|U_{n}^{k}|\right\Vert_{L^{1}L^{2}}\longrightarrow0,\ \text{for }j,k\in J_{\text{diff}}.\label{eq:Vmjk_dec_1}
\end{equation}
Let us now assume that $j\in J_{\text{comp}}$ and $k\in J_{\text{diff}}$.
We have, in a same way as before
\begin{equation}
\left\Vert|U_{n}^{j}|^{4}|U_{n}^{k}|\right\Vert_{L^{1}L^{2}}\leq\left\Vert|U_{n}^{j}|^{4}|V_{n}^{k}|\right\Vert_{L^{1}L^{2}}+\left\Vert U_{n}^{j}\right\Vert_{L^{5}L^{10}}^4\left\Vert V_{n}^{k}-U_{n}^{k}\right\Vert_{L^{5}L^{10}}.\label{eq:Vnj_Vnk_compdiff_1}
\end{equation}
On the one hand, we already saw that for $k\in J_{\text{diff}}$
\begin{equation}
\left\Vert U_{n}^{j}\right\Vert_{L^{5}L^{10}}^4\left\Vert V_{n}^{k}-U_{n}^{k}\right\Vert_{L^{5}L^{10}}\underset{n\to\infty}{\longrightarrow}0.\label{eq:Vnj_Vnk_compdiff_2}
\end{equation}
On the other hand, by Hölder inequality and change of variables
\begin{multline*}
\left\Vert|U_{n}^{j}|^{4}|V_{n}^{k}|\right\Vert_{L^{1}L^{2}}\leq\left\Vert U_{n}^{j}\right\Vert_{L^{5}L^{10}}^{3}\left\Vert V_{n}^{k}U_{n}^{j}\right\Vert_{L^{5/2}L^{5}}\\
=\left\Vert U^{j}\right\Vert_{L^{5}L^{10}}^{3}\frac{1}{\sqrt{\lambda_{k,n}}}\Big(\int\Big(\int_{r\geq1}U^{j}({t-t_{j,n}},{x})^{5}V^{k}\Big(\frac{t-t_{k,n}}{\lambda_{k,n}},\frac{x}{\lambda_{k,n}}\Big)^{5}\ dx\Big)^{1/2}dt\Big)^{2/5}\\
=\left\Vert U^{j}\right\Vert_{L^{5}L^{10}}^{3}\frac{1}{\sqrt{\lambda_{k,n}}}\Big(\int\Big(\int_{r\geq1}U^{j}(s,y)^{5}V^{k}\Big(\frac{s+t_{j,n}-t_{k,n}}{\lambda_{k,n}},\frac{y}{\lambda_{k,n}}\Big)^{5}\ dy\Big)^{1/2}ds\Big)^{2/5}.
\end{multline*}
As the above expression is uniformly continuous in $V^{k}\in L^{5}L^{10}$,
we can assume that $V^{k}$ is contiuous and compactly supported.
Then we get
\begin{equation}
\left\Vert|U_{n}^{j}|^{4}|V_{n}^{k}|\right\Vert_{L^{1}L^{2}}\lesssim\frac{1}{\sqrt{\lambda_{k,n}}}\longrightarrow0\label{eq:Vnj_Vnk_compdiff3}
\end{equation}
and thus by \eqref{eq:Vnj_Vnk_compdiff_1}, \eqref{eq:Vnj_Vnk_compdiff_2}
and \eqref{eq:Vnj_Vnk_compdiff3}
\begin{equation}
\left\Vert|U_{n}^{j}|^{4}|U_{n}^{k}|\right\Vert_{L^{1}L^{2}}\longrightarrow0,\ \text{for }j\in J_{\text{comp}},\ k\in J_{\text{diff}}.\label{eq:Vmjk_dec_2}
\end{equation}
In a similar fashion we obtain
\begin{equation}
\left\Vert|U_{n}^{j}|^{4}|U_{n}^{k}|\right\Vert_{L^{1}L^{2}}\longrightarrow0,\ \text{for }k\in J_{\text{comp}},\ j\in J_{\text{diff}}.\label{eq:Vmjk_dec_3}
\end{equation}
To conclude with the mixed term $|U_{n}^{j}|^{4}|U_{n}^{k}|$, let
us deal with the case $j,k\in J_{\text{comp}}$. Then
\begin{equation}
\left\Vert|U_{n}^{j}|^{4}|U_{n}^{k}|\right\Vert_{L^{1}L^{2}}
=\int\Big(\int_{r\geq1}U^{j}\big({t-t_{j,n}},{x}\big)^{8}U^{k}\big({t-t_{k,n}},{x}\big)^{2}\ dx\Big)^{1/2}dt.\label{eq:Vnj_Vnk_compcomp_1}
\end{equation}
By orthogonality of the parameters,
\begin{equation}
|t_{j,n}-t_{k,n}|\longrightarrow+\infty,\label{eq:Vnj_Vnk_compcomp_1.5}
\end{equation}
but, by change of variable $s=t-t_{j,n}$ we obtain from \eqref{eq:Vnj_Vnk_compcomp_1}
\[
\left\Vert|U_{n}^{j}|^{4}|U_{n}^{k}|\right\Vert_{L^{1}L^{2}}=\int\Big(\int_{r\geq1}U^{j}\big({s},{x}\big)^{8}U^{k}\big({s+t_{j,n}-t_{k,n}},{x}\big)^{2}\ dx\Big)^{1/2}ds.
\]
Again, as this expression is uniformly continuous in $(U^{j},U^{k})\in L^{5}L^{10}$,
we may assume that both are continuous and compactly supported. But
for such functions, the above expression vanishes for $n$ large enough
by \eqref{eq:Vnj_Vnk_compcomp_1.5}. Thus we have
\begin{equation}
\left\Vert|U_{n}^{j}|^{4}|U_{n}^{k}|\right\Vert_{L^{1}L^{2}}\longrightarrow0,\ \text{for }j,k\in J_{\text{comp}}.\label{eq:Vmjk_dec_4}
\end{equation}
We dealt with all the cases \eqref{eq:Vmjk_dec_1}, \eqref{eq:Vmjk_dec_2},
\eqref{eq:Vmjk_dec_3}, \eqref{eq:Vmjk_dec_4} and shown that
\begin{equation}
\forall J,\ \left\Vert\sum_{1\leq j\neq k\leq J}|U_{n}^{j}|^{4}|U_{n}^{k}|\right\Vert_{L^{1}L^{2}}\longrightarrow0.\label{eq:enj_dec_mix}
\end{equation}
Finally, by the decay of the remainder \eqref{eq:dec_znJ},
\begin{equation}
\lim_{J\to\infty}\limsup_{n\to\infty}\left\Vert|z_{n}^{J}|^{5}\right\Vert_{L^{1}L^{2}}=\lim_{J\to\infty}\limsup_{n\to\infty}\left\Vert z_{n}^{J}\right\Vert_{L^{5}L^{10}}^{5}=0,\label{eq:e_n_j_wnj}
\end{equation}
and moreover, by Minkowski and Hölder inequalities
\begin{equation}
\bigg\Vert||z_{n}^{J}|\sum_{j=1}^{J}|U_{n}^{j}|^{4}\bigg\Vert_{L^{1}L^{2}} \leq\left\Vert z_{n}^{J}\right\Vert_{L^{5}L^{10}}\sum_{j=1}^{J}\left\Vert U_{n}^{j}\right\Vert_{L^{5}L^{10}}^{4}.
\label{e_n_j_mixed_wV}
\end{equation}
By \eqref{eq:dec_znJ}, 
\begin{equation}
 \label{mixed0}
 \lim_{J\to\infty}\limsup_{n\to\infty}\bigg\Vert||z_{n}^{J}|\sum_{j=1}^{J}|U_{n}^{j}|^{4}\bigg\Vert_{L^{1}L^{2}} =0.
\end{equation} 
Combining \eqref{eq:enj_dec_mix},
\eqref{eq:e_n_j_wnj}, \eqref{e_n_j_mixed_wV} and \eqref{mixed0}, we thus proved
the $L^{1}L^{2}$ decay of the error term $e_{n}^{J}$, that is \eqref{eq:decay_en_J}.

Let us now show \eqref{eq:v_n_data}. We have, by the definition of
${\tilde u}_{n}^{J}$ \eqref{eq:def_vnj}, of the remainder \eqref{eq:def_znJ}
and of the modified profiles \eqref{eq:def_Vnj_diff}, \eqref{eq:def_Vnj_comp}\ednote{Thomas: pour faire plusieurs ligne sous un signe somme, tu peux utiliser l'instruction $\backslash$\textit{substack} {\color{pgreen}Merci!}}
\begin{equation}
\vec{\tilde u}_{n}^{j}(0)=\sum_{\substack{
j\in J_{\text{comp}}\\
j\leq J}}\vec{U}^j({-t_{j,n}})
+\sum_{\substack{
j\in J_{\text{diff}}\\
j\leq J
}}
\sigma_{\lambda_{j,n}}
\Big(\vec{V}^{j}\Big(\frac{-t_{j,n}}{\lambda_{j,n}}\Big)\Big)+\vec{w}_{n}^{J}.
\label{eq:vnJ(0)_dec}
\end{equation}
As a consequence of the definition \eqref{eq:def_Uj_comp} of $U^{j}$
for $j\in J_{\text{comp}}$, we have, in $\mathcal{H}(B^c)$, as $n$ goes
to infinity
\begin{equation}
\forall j\in J_{\text{comp}},\hspace{1em}\vec{U}^j({-t_{j,n}})=\vec{S}_{N}(-t_{j,n})\vec{\psi}^{j}\ +o_n(1).\label{eq:alpha_n_comp}
\end{equation}
Let us deal now with the first component of the diffusive profiles,
the derivative component being handled in the same fashion. For $j\in J_{\text{diff}}$,
by the definition \eqref{eq:def_Uj_diff}, this first component verifies,
in $\dot{H}^{1}$
\begin{multline}
\sigma_{\lambda_{j,n}}\left(V^{j}\Big(\frac{-t_{j,n}}{\lambda_{j,n}}\Big)\right)
=\frac{1}{\lambda_{j,n}^{1/2}}V^{j}\left(\frac{-t_{j,n}}{\lambda_{j,n}},\frac{\cdot}{\lambda_{j,n}}\right)\\
=\frac{1}{\lambda_{j,n}^{1/2}}V^{j}_L\left(\frac{-t_{j,n}}{\lambda_{j,n}},\frac{\cdot}{\lambda_{j,n}}\right)+o_n(1)
=S_{N}(-t_{j,n})\sigma_{\lambda_{j,n}}\vec{\psi}^{j}\ +o_n(1),\label{eq:alpha_n_difff}
\end{multline}
{where at the last line we have used Lemma \ref{lem:dil_lin}.}
This last expansion \eqref{eq:alpha_n_difff}, together with the similar
one for the derivative component and \eqref{eq:alpha_n_comp}, \eqref{eq:vnJ(0)_dec},
the linear profile decomposition \eqref{eq:dec_u0n} and the Strichartz estimates for the Neumann flow (Proposition \ref{prop:Strichartz}) gives \eqref{eq:v_n_data},
and ends the proof of the lemma. 
\end{proof}
The proof of the nonlinear profile decomposition follows:
\begin{proof}[Proof of Proposition \ref{prop:non_lin_prof}]
By \eqref{eq:eq_vnj} together with \eqref{eq:decay_en_J} and \eqref{eq:v_n_data},
the perturbative result of Proposition \ref{lem:perturb} gives, together
with \eqref{eq:dec_znJ}
\[
u_{n}=\tilde u_{n}^{J}+\tilde{R}_{n}^{J},
\]
with
\[
\lim_{J\rightarrow\infty}\limsup_{n\rightarrow\infty}\Vert\tilde{R}_{n}^{J}\Vert_{L^{5}L^{10}}=0.
\]
But \eqref{eq:comp_U_V} enables us to replace all the $U_{n}^{j}$ by
$V_{n}^{j}$ for $j \in J_{\text{diff}}$ in the definition \eqref{eq:def_vnj} of $\tilde u_{n}^{J}$ and
ends the proof of the nonlinear profile decomposition.
\end{proof}
We are now in position to end the proof of the theorem. Indeed, by
Proposition \ref{prop:non_lin_prof} together with (\ref{eq:tout_est_diffusif}), $u_{n}$ { is in $L^{5}L^{10}$} for $n$ large enough, and \eqref{eq:suite_u0n} is
contradicted. Therefore the assumption \eqref{eq:more_than_one_profile}
cannot hold, that is, $J=1$: there is only one non-trivial profile
in the decomposition \eqref{eq:suite_u0n}:
\begin{equation}
\label{dev_u0n}
\vec{u}_{0}^{n}=S_{N}(-t_{1,n})\sigma_{\lambda_{1,n}}\vec{\psi}^{1}+\vec{w}_{n},\hspace{1em}\Vert S_{N}(\cdot)\vec{w}_{n}\Vert_{L^{5}L^{10}}\longrightarrow0  
\end{equation} 

Let us show that it is the time-compact ($t_{1,n}=0$),
scaling-compact ($\lambda_{1,n}=1$) one.\ednote{Thomas: J'ai un peu r\'eorganis\'e la fin de cet d\'emo. Je n'ai pas colori\'e tous les changements} 

As noticed before,
as the scattering in the free space $\mathbb{R}^{3}$ is well known,
we have $V^j\in L^{5}L^{10}$ for any $j\in J_{\text{diff}}$. Therefore, if $1\in J_{\text{diff}}$,
the same proof as before yields
the decomposition:\ednote{Thomas: attention si on veut cette d\'ecomposition pour tout temps il faut que $U^1$ soit scattering (ce qui est une des raisons de ma r\'eorganisation).{\color{pgreen} en effet, merci!}} 
\begin{equation}
u_{n}(t) =\frac{1}{\lambda_{1,n}^{1/2}}V^{1}\Big(\frac{t-t_{1,n}}{\lambda_{1,n}},\frac{\cdot}{\lambda_{1,n}}\Big)+R_{n}(t)\label{eq:dec_uj-1}
\end{equation}
with 
\begin{equation}
\label{decRn}
\limsup_{n\rightarrow\infty}\Vert R_{n}\Vert_{L^{5}L^{10}}=0,
\end{equation}
proving that $u_n\in L^5L^{10}$, a contradiction. Thus $1\in J_{\text{comp}}$ i.e. $\lambda_{1,n}=1$.

It remains to eliminate the case $t_{1,n}\longrightarrow\pm\infty$.
Recall that
\begin{equation}
\label{sym_infty}
\Vert u_{n} \Vert_{L^{5}\left((-\infty,0)L^{10}\right)}\longrightarrow\infty,\hspace{1em}\Vert u_{n}\Vert_{L^{5}\left((0,+\infty)L^{10}\right)}\longrightarrow\infty.
 \end{equation} 
Let us for example assume, by contradiction, that $t_{1,n}\longrightarrow+\infty$.
This implies
$$\lim_{n\to\infty}\left\|S_N(\cdot-t_{1,n})\vec{\psi}^1\right\|_{L^5\left((-\infty,0)L^{10}\right)}=0,$$
and we obtain, by the small data well-posedness theory, that for large $n$, $u_n\in L^{5}((-\infty,0), L^{10})$ with 
$$\lim_{n\to\infty} \|u_n\|_{L^5((-\infty,0),L^{10})}=0,$$
contradicting \eqref{sym_infty}. 
The case $t_{1,n}\longrightarrow-\infty$
is eliminated in the same way.

Therefore, $\vec{u}_{0}^{n}$ writes:
\[
\vec{u}_{0}^{n}=\vec{\psi}^{1}+\vec{w}_{n},\hspace{1em}\Vert S_{N}(\cdot)\vec{w}_{n}\Vert_{L^{5}L^{10}}\longrightarrow0.
\]
Notice that, by the Pythagorean expansion \eqref{eq:pl_pyth_1} together
with its $L^{6}$ version \eqref{eq:pl_pyth_L6}, $\mathscr{E}(\vec{\psi}^{1})\leq E_{c}$,
and therefore
\[
\mathscr{E}(\vec{\psi}^{1})=E_{c}
\]
otherwise, by \eqref{eq:dec_uj-1} and the definition of $E_{c}$,
$u_{n}$ scatters. This implies, by the Pythagorean expansion again,
together with \eqref{eq:suite_u0n}
\[
\Vert\vec{w}_{n}\Vert_{\mathcal{H}(B^c)}\longrightarrow0.
\]
We take $\vec{u}_{c}$ to be this profile:

\[
\vec{u}_{c}:=\vec{\psi}^{1}.
\]
By the conservation of energy, we have $\mathscr{E}(\vec{\mathscr{S}}_{N}(t)\vec{u}_{c})=E_{c}$
for any $t$, and the same argument applied to
\[
\vec{\mathscr{S}}_{N}(t_{n})\vec{u}_{c}
\]
for any sequence $(t_{n})_{n\geq1}\in\mathbb{R}^{\mathbb{N}}$ shows
that the flow $\big\{ t\in\mathbb{R},\;\vec{ \mathscr{S}}_{N}(t)\vec{u}_{c}\big\} $
has a compact closure in $\mathcal{H}(B^c)$. Indeed this sequence satisfies the same assumptions
as $\vec{u}^0_n$ at the beginning of the proof, and will therefore have a convergent
subsequence in $\mathcal{H}(B^c)$ as well. Finally, observe that $\mathscr{E}(\vec{u}_{c})=E_{c}>0$
insures in particular that $\vec{u}_{c} \neq \vec{0}$.
\end{proof}

\section{Rigidity}
\label{section:rigidity}
In this section we prove:
\begin{thm}
\label{thm:rigidity}
 Let $(u_0,u_1)\in \mathcal{H}(B^c)$, radial, and $u(t)=\mathscr{S}_N(t)(u_0,u_1)$ be a solution of the energy critical defocusing wave equation outside the unit ball with Neumann boundary conditions (\ref{NLWrad})-(\ref{NBDrad})-(\ref{IDrad}). Assume that $u$ is global and that 
 $$ K=\Big\{\vec{u}(t),\; t\in \mathbb{R}\Big\}$$
has compact closure in $\mathcal{H}(B^c)$. Then $u = 0$.
 \end{thm}
The proof follows the lines of the proof of \cite{DuKeMe13}. 
\subsection{Preliminaries}
We will use the following asymptotic energy property for the wave equation on $\mathbb{R}^3$:
\begin{prop}
\label{prop:channels}
Let $R>0$.
 Let $(v_0,v_1)\in \mathcal H (\RR^3)$ and $v=S_{\RR^3}(v_0,v_1)$ be the solution of the linear wave equation on $\RR^3$ with initial data $(v_0,v_1)$. Then
 $$ \sum_{\pm}\lim_{t\to\pm \infty} \int_{R+|t|}^{+\infty} |\partial_{t,r}(rv(t,r))|^2\,dx
 =\int_R^{+\infty} \left( \partial_{r}(rv_0) \right)^2+r^2v_1^2\,dr.$$
\end{prop}
We omit the easy proof, which relies on the equation $(\partial_t^2-\partial_r^2)u(t,r)=0$.
We note that by integration by parts,
\begin{equation}
\label{IPP}
\int_R^{+\infty} \left( \partial_{t,r}(ru_0) \right)^2\,dr+Ru_0^2(R)=\int_R^{+\infty} \left( \partial_{t,r}(u_0) \right)^2r^2\,dr.
\end{equation} 
\begin{prop}
\label{prop:Z}
There exists $\mathfrak{z}>0$ and a radial, $C^{\infty}$ function $Z=Z(|x|)$ on $\{x\in \RR^3,\; |x|>\mathfrak{z}\}$ such that\begin{gather}
 \label{eqZ}
 \Delta Z=Z^5\text{ for }r>\mathfrak{z}\\
 \label{ZlargeR}
 \left|r Z'(r)+\frac 1r\right|+\left|Z(r)-\frac{1}{r}\right|\leq \frac{C}{r^3}\\
  \label{blowupZ}
  \lim_{r\to\zeta^+}|Z(r)|=+\infty\\
  \label{Zmonotonic}
  Z'(r)\neq 0\text{ for }r>\mathfrak{z}.
 \end{gather}
\end{prop}
\begin{proof}
The existence of $\mathfrak{z}$ and $Z$ satisfying \eqref{eqZ}, \eqref{ZlargeR} and \eqref{blowupZ} is proved in 
\cite[Proposition 4.1]{DuyckaertsYang18} and we omit it. 

To prove \eqref{Zmonotonic}, we argue by contradiction. Assume that $Z'(R)=0$ for some $R>\mathfrak{z}$. Multiplying equation \eqref{eqZ} by  $Z$, integrating by parts for $r>R$ and using the boundary condition $Z'(R)=0$, we obtain 
$$\int_{\{|x>R\}}|\nabla Z|^2\,dx+\int_{\{|x>R\}}|Z|^6\,dx=0.$$
This proves that $Z(r)=0$ for almost every $r>R$, contradicting \eqref{ZlargeR}. 
\end{proof}
\begin{remark}
 \label{rem:Z}
 Let $\ell\in \RR\setminus\{0\}$ and
 $$Z_{\ell}=\frac{1}{\ell}Z\left( \frac{r}{\ell^2}\right).$$
 Then \eqref{eqZ}, \eqref{blowupZ} and \eqref{Zmonotonic} hold with $Z$ replaced by $Z_{\ell}$ and $\mathfrak{z}$ by $\ell^2\mathfrak{z}$, and there exists a constant $C_{\ell}$ such that
\begin{equation}
   \label{Zell_largeR}
 \left|r Z'_{\ell}(r)+\frac{\ell}{r}\right|+\left|Z_{\ell}(r)-\frac{\ell}{r}\right|\leq \frac{C_{\ell}}{r^3}.
 \end{equation} 
\end{remark}

\subsection{Proof of Theorem \ref{thm:rigidity}}
\setcounter{step}{0}
\begin{step} \ednote{\color{pgreen} Changé les solutions $\mathbb R^3$: $\tilde u$ etc... pour $v$}
 \label{step:small}
 { Let $(u_0, u_1)\in\mathcal H(B^c)$ be as in Theorem \ref{thm:rigidity}.} Let $\eps>0$ be a small parameter to be specified. In all the proof we fix $R_{\eps}>1$ such that 
 \begin{equation}
  \label{defReps}
  \int_{R_{\eps}}^{+\infty}\left((\partial_ru_0)^2+u_1^2\right)r^2\,dr\leq \eps.
 \end{equation} 
 In this step, we prove
 \begin{equation}
  \label{R4}
  \forall R\geq R_{\eps},\quad
\int_R^{+\infty} \left( \partial_{r}(r u_0) \right)^2+r^2u_1^2\,dr\leq CR^5u_0^{10}(R).
 \end{equation} 
 Let $R\geq R_{\eps}$.
 We define the radial functions $v_0\in \dot{H}^1(\RR^3)$, $v_1\in L^2(\RR^3)$ as follows:
 \begin{equation}
 \label{tildeu}
 \begin{cases}
  (v_0,v_1)(r)=(u_0,u_1)(r)&\text{ if }r>R\\
  (v_0,v_1)(r)=(u_0(R),0)&\text{ if }r\in (0,R).
 \end{cases} 
 \end{equation}   \ednote{\color{pgreen}Quid de $u_1$ pour $|r|<1$ ?}
We let $v(t)=\mathscr{S}_{\RR^3}(t)(v_0,v_1)$ be the solution to the quintic wave equation on $\RR^3$ with initial data $(v_0,v_1)$, and $v_L(r)=S_{\RR^3}(v_0,v_1)$ be the corresponding solution to the free wave equation. We note that by final speed of propagation 
$$ v(t,r)=u(t,r),\quad r>R+|t|. $$ 
By the small data theory, since $\eps$ is small,
\begin{equation}
 \label{R2}
 \sup_{t\in \RR} \left\|\vec{v}(t)-\vec{v}_L(t)\right\|_{\dot{H}^1\times L^2}\leq C\left\|(v_0,v_1)\right\|^5_{\dot{H}^1\times L^2}.
\end{equation}
By Proposition \ref{prop:channels},
\begin{equation}
 \label{R3}
 \sum_{\pm}\lim_{t\to\pm \infty} \int_{R+|t|}^{+\infty} |\partial_{t,r}(rv_L(t,r))|^2\,dr
 =\int_R^{+\infty} \left( \partial_{r}(r u_0) \right)^2+u_1^2\,dr. 
\end{equation} 
By \eqref{R2}, and finite speed of propagation
$$\left|\int_{R+|t|}^{+\infty} \big|\partial_{t,r}(rv_L(t,r))-\partial_{t,r}(r u(t,r))\big|^2\,dr\right|\leq 
C\left(\int_R^{+\infty} \big((\partial_r u_0)^2+u_1^2\big)r^2 dr\right)^5.
$$
Combining with \eqref{R3} and using that by the compactness of the closure of $K$ in $\mathcal{H}(B^c)$
$$\lim_{t\to\pm \infty}\int_{R+|t|}^{+\infty} |\partial_{t,r}(r u(t,r))|^2\,dr=0,$$
we deduce
$$
\int_R^{+\infty} \left( \partial_{r}(r u_0) \right)^2+u_1^2\,dr\leq C\left(\int_R^{+\infty} \left(\partial_r u_0)^2+u_1^2\right)r^2 dr\right)^5.
$$
Combining with the integration by parts formula \eqref{IPP} and the smallness of $\eps$, we deduce \eqref{R4}.
\end{step}
\begin{step}
\label{step:limit}
 In this step we prove that there exists $\ell\in \RR$ and $C>0$ such that for large $R$,
 \begin{equation} 
  \label{R5}
  \left|u_0(r)-\frac{\ell}{r}\right|\leq \frac{C}{r^3},\quad \int_r^{+\infty}\rho^2u_1(\rho)\,d\rho\leq \frac{C}{r^5}.
 \end{equation} 
 First fix $R$ and $R'$ such that $R_{\eps}\leq R\leq R'\leq 2R$. Letting $\zeta_0(r)=ru_0(r)$\ednote{\color{pgreen}Changé $v_0$ pour $\zeta_0$ dans ce qui suit pour être coherent avec les sections precedentes. Le $\zeta$ associé a $Z$ a été change en $\mathfrak{z}$}, we have, using Cauchy-Schwarz, then Step \ref{step:small}
 \begin{equation}
  \label{R6}
  \left|\zeta_0(R)-\zeta_0(R')\right|\leq \int_{R}^{R'}|\partial_r\zeta_0(r)|\,dr\leq \sqrt{R}\sqrt{\int_R^{R'}(\partial_r \zeta_0)^2dr}\leq \frac{1}{R^2}\zeta_0^5(R).
 \end{equation}
 Since by the definition \eqref{defReps} of $R_{\eps}$ and the integration by parts formula \eqref{IPP} one has
 \begin{equation}
  \label{R7}
  \frac{1}{R}\zeta_0^2(R)\leq \eps,
 \end{equation} 
 we deduce from \eqref{R6}:
 \begin{equation}
  \label{R8}
  \left|\zeta_0(R)-\zeta_0(R')\right|\leq \eps^2 \zeta_0(R).
 \end{equation} 
 We apply this inequality between $2^kR$ and $2^{k+1}R$ for $k\in \mathbb{N}$ and a fixed $R\geq R_{\eps}$. This yields
 $$\left|\zeta_0\left( 2^{k+1}R\right)-\zeta_0\left( 2^kR\right)\right|\lesssim \eps^2\left|\zeta_0\left( 2^kR\right)\right|$$
 and thus
 $$\left|\zeta_0\left(2^{k+1}R\right)\right|\leq (1+C\eps^2)\left|\zeta_0\left( 2^kR \right)\right|.$$
 We deduce, by an easy induction:
 $$ \left|\zeta_0\left( 2^kR \right)\right|\leq (1+C\eps^2)^{k}|\zeta_0(R)|$$
 Combining with \eqref{R6} we obtain
 \begin{equation}
 \label{R9}
 \left|\zeta_0(2^kR)-\zeta_0(2^{k+1}R)\right|\lesssim \left( 1+C\eps^2 \right)^{5k}|\zeta_0(R)|^5\frac{1}{2^{2k}R^{2}}. 
 \end{equation}
 Chosing $\eps$ small, so that $(1+C\eps^2)^5<4$, we see that $\sum \left|\zeta_0(2^kR)-\zeta_0(2^{k+1}R)\right|$ converges, and thus that $\zeta_0(2^kR)$ has a limit $\ell(R)$ as $k\to \infty$. Summing \eqref{R9} over all $k\geq k_0$, we obtain
 \begin{equation}
  \label{R10}
  \left|\zeta_0(2^{k_0}R)-\ell(R)\right|\lesssim \frac{1}{R^{2}}\frac{1}{(1+c_{\eps})^{k_0}}[\zeta_0(R)|^5,
 \end{equation} 
 for some constant $c_{\eps}>0$. Combining with \eqref{R6}, we see that 
 $$\lim_{r\to\infty}\zeta_0(r)=\ell(R),$$
 and in particular the limit $\ell(R)$ does not depend on $R$. We will simply denote it by $\ell$.  By \eqref{R10} at $k_0=1$, since $\zeta_0$ is bounded
 \begin{equation}
  \label{R11}
  \left|\zeta_0(R)-\ell\right|\lesssim_{\zeta_0} \frac{1}{R^2}, 
 \end{equation} 
 which yields the first inequality in \eqref{R5}.Combining with step 1, we obtain the second inequality in \eqref{R5}.
\end{step}
\begin{step}
\label{step:ell0}
 In this step, we assume $\ell=0$ and prove that $(u_0,u_1)\equiv (0,0)$. Indeed by \eqref{R8}, if $R\geq R_{\eps}$ and $k\in \mathbb{N}$, 
 $$ \left|\zeta_0\left(2^{k+1}R\right)\right| \geq (1-C\eps^2)\left|\zeta_0\left( 2^kR \right)\right|.$$
 Hence by induction on $k$,
 $$ \left|\zeta_0\left( 2^kR \right)\right| \geq (1-C\eps^2)^{k} |\zeta_0(R)|.$$
 Since by the preceding step and the assumption $R=0$, $|\zeta_0(2^kR)|\lesssim 1/{2^kR}^2$, we  deduce, chosing $\eps$ small enough and letting $k\to\infty$ that $\zeta_0(R)=0$. Combining with \eqref{R4} we deduce
 \begin{equation*}
  R\geq R_{\eps}\Longrightarrow \int_{R}^{+\infty} (\partial_r \zeta_0)^2+u_1^2(r)\,dr=0,
 \end{equation*} 
 that is $u_0(r)$ and $u_1(r)$ are $0$ for almost every $r\geq R_{\eps}$. Going back to the definition of $R_{\eps}$ we see that we can choose any $R_{\eps}>1$, which concludes this step.
\end{step}
\begin{step}
\label{step:Zell}
We next assume $\ell \neq 0$. Let $Z_{\ell}$ be as in Remark \ref{rem:Z}. In this step we prove that $(u_0-Z_{\ell},u_1)$ has  a bounded support. Let $f=u-Z_{\ell}$. Then
\begin{equation}
 \label{eqf}
 \left\{
 \begin{aligned}
\partial_t^2 f-\Delta f&=D_{\ell}(f):= \sum_{k=1}^5 \binom{5}{k}Z_{\ell}^{5-k}f^k.\\
\vec{f}_{\restriction t=0}&=(f_0,f_1):=\left( u_0-Z_{\ell},u_1 \right),
 \end{aligned}
\right.
\end{equation}
For $\eps>0$ small, we fix $R_{\eps}'\gg 1$ such that 
\begin{gather} 
 \label{defReps1}
 \int_{R_{\eps}'}^{+\infty} \left(|\partial_rf_0(r)|^2+|f_1(r)|^2\right)r^2dr\leq \eps^2\\
 \label{defReps2}
 \int_{\RR} \left( \int_{R_{\eps}'+|t|}^{+\infty} Z_{\ell}^{10}(r)r^2\,dr \right)^{\frac 12}dt\leq \eps^{5}.
\end{gather}
Let $f_{L}$ be the solution of $\partial_t^2f_L=\Delta f_L$ with 
$$\vec{f}_{L\restriction t=0}=(\tilde{f}_0,\tilde{f}_1),$$
where $(\tilde{f}_0,\tilde{f}_1)$ coincides with $(f_0,f_1)$ for $r>R_{\eps}'$ and is defined as in \eqref{tildeu}. Using \eqref{eqf} and the assumptions \eqref{defReps1} and \eqref{defReps2} on $R_{\eps}'$, we obtain 
\begin{equation}
\label{f_tf}
\sup_{t\in \RR} \left\|\indic_{\{|x|>|t|+R_{\eps}'\}}\big|\nabla_{t,x}(\tilde{f}(t)-\tilde{f}_L(t))\big|\,\right\|_{L^2}\lesssim \eps^4\Big\|(\tilde{f}_0,\tilde{f}_1)\Big\|_{\dot{H}^1\times L^2}.
\end{equation}
Let $R\geq R_{\eps}'$.
Using that by Proposition \ref{prop:channels},
$$\sum_{\pm}\lim_{t\to\pm\infty}\int_{R}^{+\infty}\left( \partial_{t,r}\big(r\tilde{f}_L(t,r)\big) \right)^2dr\gtrsim \int_{R}^{+\infty}
\left(\left( \partial_r(r\tilde{f}) \right)^2+r^2\tilde{f}_1^2\right)dr,
$$
and since 
$$\sum_{\pm}\lim_{t\to\pm\infty}\int_{R}^{+\infty}\left( \partial_{t,r}\big(r\tilde{f}(t,r)\big) \right)^2dr=0,$$
we deduce from \eqref{f_tf}
$$\eps^{8}\int_R^{+\infty}\left( (\partial_rf_0)^2+f_1^2 \right)r^2dr\gtrsim\int_{R}^{+\infty}\left(\big(\partial_r(rf_0)\big)^2+r^2f_1^2\right)dr,$$
and thus
\begin{equation}
\label{boundf0}
\eps^8Rf_0^2(R)\gtrsim \int_R^{\infty}\left(\big(\partial_r(rf_0)\big)^2+r^2f_1^2\right)dr. 
\end{equation} 
Letting $g_0=rf_0$, we deduce by Cauchy-Schwarz that for $R\geq R_{\eps}'$, $k\in \NN$,
$$\left|g_0\left(2^{k+1}R\right)-g_0\left(2^{k}R\right)\right|\lesssim \int_{2^kR}^{2^{k+1}R} |\partial_r g_0|dr\lesssim \eps^4\left|g_0(2^kR)\right|.$$
This yields by an easy induction $|g_0(2^{k}R)|\geq \left(1-C\eps^4\right)^k|g_0(R)|$, where $C>0$ is a constant which is independent of $\eps$. Since by Step \ref{step:limit},
$$\frac{C}{\left( 2^kR\right)^2}\geq \left|g_0\left( 2^kR \right)\right|,$$
we obtain choosing $\eps$ small enough that $g_0(R)=0$ for large $R$. Combining with \eqref{boundf0}, we deduce that $(f_0(r),f_1(r))=0$ a.e.\,for large $R$, concluding this step.
\end{step}
\begin{step}
\label{step:contradiction}
 In this step we still assume $\ell\neq 0$ and deduce a contradiction. We let
 $$ \rho=\inf\left\{ R>c\;:\; \int_{R}^{+\infty}\left((\partial_r f_0)^2+f_1^2\right)r^2dr=0\right\}$$
 and prove that $\rho= \max(1,\mathfrak{z} \ell^2)$, i.e.\,that $u_0(r)=Z_{\ell}(r)$ almost everywhere for $r>\max(1,\mathfrak{z}\ell^2)$. If $\mathfrak{z}\ell^2 \geq 1$, we deduce
 $$\lim_{r\to \mathfrak{z}\ell^2}|u_0(r)|=+\infty,$$
 a contradiction with the radial Sobolev embedding theorem. If $\mathfrak{z}\ell^2\leq 1$, we obtain
 $$u_0(r)=Z_{\ell}(r)$$
 for all $r>1$. Translating the solution in time, the same proof yields that for all $t$ in the domain of definition of $u$,
 \begin{equation}
 \label{u_stat}
 u(t,r)=Z_{\ell}(r),  
 \end{equation} 
 a contradiction with the Neumann boundary condition, as given by Lemma \ref{lem:interpr_neumann}. Note that by finite speed of propagation, the limit $\ell$ in \eqref{u_stat} is independent of $t$.
 
 To prove that $\rho=\max(1,\mathfrak{z}\ell^2)$, we argue by contradiction, assuming $\rho>\max(1,\mathfrak{z} \ell^2)$. By the preceding step and finite speed of propagation, the essential support of $f$ is included in $\{r\leq \rho+|t|\}$. Thus $f$ is solution of 
\begin{equation*}
 \left\{
 \begin{aligned}
\partial_t^2 f-\Delta f&=\indic_{\{|x|\leq \rho+|t|\}}D_{\ell}(f).\\
\vec{f}_{\restriction t=0}&=(f_0,f_1):=\left( u_0-Z_{\ell},u_1 \right),
 \end{aligned}
\right.
\end{equation*} 
Fix $R_{\eps}''\in (1,\rho)$ such that, 
\begin{gather*} 
 \int_{R_{\eps}''}^{+\infty} \left(|\partial_rf_0(r)|^2+|f_1(r)|^2\right)r^2dr\leq \eps^2\\
 \int_{\RR} \left( \int_{R_{\eps}''+|t|}^{\rho+|t|} Z_{\ell}^{10}(r)r^2\,dr \right)^{\frac 12}dt\leq \eps^{5}.
\end{gather*}
The same argument as in the preceding step, replacing $R_{\eps}'$ by $R_{\eps}''$, yields that $(f_0,f_1)=0$ for almost every $r>R_{\eps}''$, which contradicts the definition of $\rho$. The proof is complete.
\end{step}

\qed

We are now in position to conclude
\begin{proof}[Proof of Theorem \ref{thm:main}]
By contradiction, assume that $E_c$, as defined by (\ref{eq:Ec}), is finite. Then Theorem \ref{th:conc-comp} shows that there exists a solution $\vec{u}_c$
to (\ref{NLWrad})-(\ref{NBDrad})-(\ref{IDrad}) such that $\big\{ \vec{u}_c(t), \ t\in\mathbb R \big\}$ has a compact closure in $\mathcal H{ (B^c)}$, but by Theorem \ref{thm:rigidity}, such a solution cannot exist. Thus
$E_c = + \infty$, and by Proposition \ref{prop:perturb_scat}, all the solutions of (\ref{NLWrad})-(\ref{NBDrad})-(\ref{IDrad}) scatter.
\end{proof}

\section{Focusing case}
\label{section:focusing}
In this section we sketch the proofs of Theorems \ref{thm_foc1} and \ref{thm_foc2}. Subsection \ref{subs:trapping} is dedicated to the proof of a trapping property for solutions below the energy of the $\RR^3$ ground state $W$ that is important in the proof of both of these results. Subsection \ref{subs:foc1} concerns Theorem \ref{thm_foc1} and Subsection \ref{subs:foc2} Theorem \ref{thm_foc2}. Finally, in Subsection \ref{subs:geom}, we comment on the assumptions of these two theorems, and prove that the exact analog of Theorem \ref{thm_foc1} is not true when $\RR^3\setminus B(0,1)$ is replaced by a more general domain. 
\subsection{Trapping by the energy}
\label{subs:trapping}
Recall that
$$W(x)=\frac{1}{\left( 1+\frac{|x|^2}{3}\right)^{\frac{1}{2}}}$$
is the ground state of the focusing critical wave equation on $\RR^3$. If $(f,g)\in \mathcal H (\RR^3)$, we denote by 
$$\mathscr{E}_{\RR^3}(f,g)=\frac{1}{2}\int_{\RR^3} |\nabla f|^2+\frac{1}{2}\int_{\RR^3} |g|^2-\frac{1}{6}\int_{\RR^3}|f|^6.$$
\begin{prop}
 \label{prop:trapping}
 Let $u$ be a solution of \eqref{NLWrad_foc} with Neumann boundary condition \eqref{NBDrad} and initial data \eqref{IDrad}. Let $I$ be its maximal interval of existence.
 Assume $\mathscr{E}(u_0,u_1)<\mathscr{E}_{\RR^3}(W,0)$. Then the sign of $\int_{  B^c} |\nabla u(t)|^2-\int_{\RR^3} |\nabla W|^2$ is independent of $t\in I$, and there exists $\delta>0$ depending only on $\mathscr{E}(u_0,u_1)$ such that 
 \begin{equation}
 \label{trapping1}
 \forall t\in I,\quad \left|\int_{ B^c} |\nabla u(t,x)|^2-\int_{\RR^3} |\nabla W(x)|^2dx\right|\geq \delta.
\end{equation} 
\end{prop}
\begin{proof}
For $(f,g)\in \mathcal H$, we denote { by} $(\tilde f, \bar g) := \vec{\mathcal P}(f,g)$, the extension of
$(f,g)$ to $\mathcal H (\mathbb R ^3)$ by $(f(1), 0)$, as defined in definition \ref{def:prolong}. Observe that
$(\tilde f, \bar g)$ verifies
$$\int_{\RR^3}\big|\nabla \tilde{f}\big|^2=\int_{ B^c}|\nabla f|^2\text{ and }\int_{\RR^3} \tilde{f}^6\geq \int_{ B^c}f^6,$$
$$\int_{\RR^3} |\bar{g}|^2=\int_{ B^c}|g|^2,$$
and
\begin{equation}
\label{lower_energy}
\mathscr{E}_{\RR^3}\left( \tilde{f},\bar{g} \right)\leq \mathscr{E}\left( f,g \right). 
\end{equation} 
Let $u$ satisfy the assumptions of Proposition \ref{prop:trapping}. Then by conservation of the energy and \eqref{lower_energy},
$$ \forall t\in I,\quad \mathscr{E}_{\RR^3}\left( \tilde{u}(t),\overline{\partial_t u}(t) \right)\leq \mathscr{E}(u_0,u_1)<\mathscr{E}_{\RR^3}(W,0).$$
The conclusion of the proposition then follows from the variational properties of the ground-state $W$ on $\RR^3$, see e.g.\,\cite[Lemma 3.4]{KeMe06}.
\end{proof}

\subsection{Scattering}
\label{subs:foc1}
{
Note that by Proposition \ref{prop:trapping} and the radial Sobolev inequality (see Remark \ref{R:global}), any solution of (\ref{NLWrad_foc})-(\ref{NBDrad})-(\ref{IDrad}) that satisfies
$\mathscr{E}(u_0,u_1)<\mathscr{E}_{\RR^3}(W,0)$, $\int_{ B^c}|\nabla u_0|^2<\int_{\RR^3}|\nabla W|^2$
is global.
}

Using Proposition \ref{prop:trapping}, the proof of Theorem \ref{thm_foc1} follows exactly the same lines as the proof of Theorem \ref{thm:main}. 

Recall that according to \cite{KeMe08}, any solution of the quintic focusing wave equation on $\RR^3$ with initial data $(v_0,v_1)\in (\dot{H}^1\times L^2)\left( \RR^3 \right)$ such that 
$$ \int_{\RR^3}|\nabla v_0|^2<\int_{\RR^3}|\nabla W|^2\text{ and }\mathscr{E}_{\RR^3}(v_0,v_1)<\mathscr{E}_{\RR^3}(W,0)$$
scatters to a linear solution.

Arguing by contradiction and using the arguments of Sections \ref{section:dilating}, \ref{section:profiles},  and \ref{section:critical}, we see that it is sufficient to prove:
\begin{thm}
\label{thm:rigidity_foc}
 Let $(u_0,u_1)\in \mathcal{H}(B^c)$, radial, and $u(t)$ be a solution of the energy critical focusing wave equation outside the unit ball with Neumann boundary conditions 
(\ref{NLWrad_foc})-(\ref{NBDrad})-(\ref{IDrad}). Assume that $u$ is global and that 
 $$ K=\Big\{\vec{u}(t),\; t\in \mathbb{R}\Big\}$$
has compact closure in $\mathcal{H}(B^c)$. Then $u\equiv 0$.
 \end{thm}
Note that it would be sufficient to prove Theorem \ref{thm:rigidity_foc} with the additional assumptions $\mathscr{E}(u_0,u_1)<\mathscr{E}_{\RR^3}(W,0)$, $\int_{ B^c}|\nabla u_0|^2<\int_{\RR^3}|\nabla W|^2$, but these assumptions are not needed to obtain the conclusion of the theorem.

The proof of Theorem \ref{thm:rigidity_foc} is the same as the proof of the Theorem \ref{thm:rigidity}  in Section \ref{section:rigidity}, except that in Steps \ref{step:Zell} and \ref{step:contradiction} the solution $Z_{\ell}$ of the elliptic equation $\Delta Z_{\ell} =Z_{\ell}^5$ must be replaced by the solution $W_{\ell}$ of the elliptic equation $-\Delta W_{\ell}=W_{\ell}^5$, where
\begin{equation}
\label{defWell}
W_{\ell}(x)=\frac{\sqrt{3}}{\ell}W\left( \frac{3x}{\ell^2} \right)=\frac{\sqrt 3}{\ell\left( 1+\frac{3|x|^2}{\ell^4} \right)^{1/2}}, 
\end{equation} 
so that 
$$ \left|W_{\ell}(x)-\frac{\ell}{|x|}\right|\lesssim \frac{1}{|x|^3},\quad |x|\gg 1.$$
Also, since $W_{\ell}(x)$ is defined for all $x\in \RR^3$, whereas $Z_{\ell}(x)$ is only defined for $r>\mathfrak{z}\ell^2$, we must replace $\max(1,\mathfrak{z} \ell^2)$ everywhere in these two steps of the proof by $1$. The key point to obtain the contradiction is that $\partial_r W_{\ell}(1)\neq 0$ for any $\ell\neq 0$, i.e. that $W_{\ell}$ is not a stationary solution of the focusing wave equation on $ B^c$ with Neumann boundary condition, which can be easily checked on the explicit formula \eqref{defWell} .
\subsection{Blow-up}
\label{subs:foc2}
Using Proposition \ref{prop:trapping}, the proof of Theorem \ref{thm_foc2} is very close to the proof of its analog on the whole space $\RR^3$, see Theorem 3.7 and the proof of Theorem 1.1, (ii) in section 7 of \cite{KeMe08}. Let us mention that this argument is inspired by the work of H.A.~Levine \cite{Levine74}.

Let us first assume that $u_0\in H^1( B^c)=\dot{H}^1( B^c)\cap L^2( B^c)$. Using the equation satisfied by $u$, one sees that $u(t)\in L^2( B^c)$ for all $t$ and, denoting $y(t)=\int_{ B^c}u^2(t,x)dx$, that
$$ y'(t)=2\int_{ B^c} u\partial_tu,\quad y''(t)=2\int_{ B^c} u^6-2\int_{ B^c}|\nabla u|^2+\int_{ B^c} (\partial_tu)^2.$$
Note that we have used the boundary condition $\partial_n u_{\restriction\partial  B^c}=0$ which implies $\int_{ B^c}u\Delta u=-\int_{ B^c} |\nabla u|^2$.

Recall that $\mathscr{E}_{\RR^3}(W,0)=\frac{1}{3}\int_{\RR^3}|\nabla W|^2$.
As in the proof of Theorem 3.7 of \cite{KeMe08}, one can write, for $t$ in the domain of existence of $u$,
\begin{align*}
y''(t)&=-12\mathscr{E}(u_0,u_1)+4\int_{ B^c}|\nabla u|^2+8\int_{ B^c} (\partial_tu)^2\\
&=8\int_{ B^c} (\partial_tu)^2+4\int_{ B^c} |\nabla u|^2-4\int_{\RR^3}|\nabla W|^2+12\mathscr{E}_{\RR^3}(W,0)-12\mathscr{E}(u_0,u_1)\\
&\geq 8\int_{ B^c}(\partial_tu)^2+\delta_0,
\end{align*}
where $\delta_0= 12\mathscr{E}_{\RR^3}(W,0)-12\mathscr{E}(u_0,u_1) > 0$ and we have used that by Proposition \ref{prop:trapping}, $\int_{ B^c}|\nabla u(t)|^2>\int_{\RR^3}|\nabla W|^2$ for all $t$.

The end of the proof that $u$ blows up in finite time is exactly as the end of the proof of Theorem 3.7, p.165 of \cite{KeMe08} and we omit it.

To treat the general case $u_0\in \dot{H}^1( B^c)$ one should use a localized version of $\int_{ B^c} u^2(t)$. These bring out new terms in the preceding computation, that can be controled using finite speed of propagation. We refer to \cite[p.205-206]{KeMe08} for the details.

\subsection{Comments on the assumptions}\label{subs:geom}
Consider the nonlinear focusing wave equation \eqref{NLW} with $\iota=-1$, and Neumann boundary condition \eqref{NBD} in a general open domain $\Omega$ of $\RR^3$. We claim that the analog of Theorem \ref{thm_foc1} does not hold in general. Indeed, first consider the case of a half-plane:
$$\Omega=\left\{ (x_1,x_2,x_3)\in \RR^2,\; x_1>0\right\}.$$
Let $w$ be the restriction of $W$ to $\Omega$. Then $w$ is a solution of $-\Delta w=w^5$. Since $W$ is radial, $w$ satisfies in addition the Neumann boundary condition \eqref{NBD}. This yields a non-scattering solution $w$ of \eqref{NLW}, \eqref{NBD} such that 
$$ \int_{\Omega} |\nabla w|^2=\frac{1}{2}\int_{\RR^3}|\nabla W|^2,\quad \mathscr{E}(\vec{w}(0))=\frac{1}{2}\mathscr{E}_{\RR^3}(W,0),$$
which proves that one cannot generalise Theorem \ref{thm_foc1} in this setting. Similarly, for $\eps>0$, the solution $w_{\eps}$ of \eqref{NLW}, \eqref{NBD} with initial data $((1+\eps)w,0)$ blows up in finite time by \cite{KeMe08}. This solution satisfies
$$ \int_{\Omega} |\nabla w|^2=\frac{(1+\eps)^2}{2}\int_{\RR^3}|\nabla W|^2,\quad \mathscr{E}(\vec{w}(0))<\frac{1}{2}\mathscr{E}_{\RR^3}(W,0),$$
which shows that the assumptions $\mathscr{E}(u_0,u_1)<\mathscr{E}_{\RR^3}(W,0)$, $\int_{\Omega} |\nabla u_0|^2<\int_{\RR^3} |\nabla W|^2$ is not sufficient to ensure global existence on the half-plane.

We now give a similar example when $\Omega$ is an exterior domain. Assume that $\Omega=\RR^3\setminus K$, where $K$ is bounded subset of $\RR^3$ with a smooth boundary $\partial K=\partial \Omega$ containing a portion of a plane. Without loss of generality, we can assume (translating and rescaling $\Omega$):
$$ \{0\}\times [-1,+1]^2\subset \partial \Omega, \quad  B(0,1)\cap \{ x_1>0\} \subset \Omega.$$
According to \cite{KrScTa09}, for all $\eps>0$, there exists a radial solution $z$ of the focusing critical wave equation on $\RR^3$, blowing-up in finite time $T>0$ and such that
$$ \limsup_{t\to T} \int_{\RR^3} \left|\nabla\left(z(t,x)-\frac{1}{t}W\left( \frac{x}{t^2} \right)\right)\right|^2+(\partial_tz(t,x))^2\,dx \leq \eps,\quad \mathscr{E}(\vec{z}(0))\leq \mathscr{E}(W,0)+\eps.$$
using finite speed of propagation, time translating and rescaling the solution, we can assume that the support of $\vec{z}(t)$ is included in $B(0,1)$ for all $t\in [0,T)$. The restriction $u$ of $z$ to $x_1>0$ is then a solution of \eqref{NLW}, \eqref{NBD}, \eqref{ID} that satisfies
$$ \mathscr{E}(u_0,u_1)\leq \frac{1}{2}\mathscr{E}_{\RR^3}(W,0)+\eps,\quad \limsup_{t\to T} \int_{\Omega} |\nabla u(t)|^2+\int_{\Omega} (\partial_tu(t))^2\leq \frac{1}{2}\int_{\RR^3} |\nabla W|^2+\eps,$$
proving that a generalization of Theorem \ref{thm_foc1} is hopeless in this setting also. 

In view of this example, we conjecture that Theorem \ref{thm_foc2} cannot be either generalised to other geometries, and that the radiality assumptions in Theorems \ref{thm_foc1} and \ref{thm_foc2} is also necessary. More precisely, a natural conjecture is that the energy threshold to ensure energy trapping and a blow-up scattering/dichotomy in the case of Neumann boundary condition is exactly $\frac{1}{2}\mathscr{E}_{\RR^3}(W,0)$. This is of course the case when $\Omega$ is a half-plane, since one can then use the result on $\RR^3$ after extending the solution by symmetry to the whole space.

\subsection*{Acknowledgements} The authors thank Fabrice Planchon for interesting discussions about the problem.



\bibliographystyle{amsalpha}
\bibliography{refs}

\end{document}